\documentclass[a4paper, 12pt]{amsart}
\usepackage[cp1251]{inputenc}
\usepackage[english]{babel}
\usepackage{amsmath,amssymb,a4wide,latexsym,amscd,amsthm}
\usepackage{graphics}

\theoremstyle{plain}
\newtheorem{theorem}{Theorem}[section]
\newtheorem{lemma}{Lemma}[section]

\theoremstyle{definition}
\newtheorem{example}{Example}[section]

\newtheorem{remark}{Remark}[section]

\numberwithin{equation}{section}

\begin{document}

\title[]
{When is the sum of complemented subspaces complemented?}

\author{Ivan Feshchenko}

\address{Taras Shevchenko National University of Kyiv,
Faculty of Mechanics and Mathematics, Kyiv, Ukraine}

\email{ivanmath007@gmail.com}



\begin{abstract}
We provide a sufficient condition for the sum of a finite number of complemented subspaces of a Banach space to be complemented.
Under this condition a formula for a projection onto the sum is given.
We also show that the condition is sharp (in a certain sense).
As applications, we get
(1) sufficient conditions for the complementability of sums of marginal subspaces in $L^p$ and
sums of tensor powers of subspaces in a tensor power of a Banach space
and (2) quantitative results on stability of the complementability property of the sum of linearly independent subspaces.
\end{abstract}

\maketitle

\section{Introduction}

\subsection{Complemented subspaces in Banach spaces}
Let $X$ be a (complex or real) Banach space.
By a subspace of $X$ we will mean a linear subset of $X$.
Let $M$ be a subspace of $X$.
$M$ is said to be complemented in $X$ if there exists a subspace $N$ (a complement) such that $X$ is the topological direct sum of $M$ and $N$.
This means that the sum operator $S:M\times N\to X$ defined by $S(x,y)=x+y$, $x\in M, y\in N$ is an isomorphism (of normed linear spaces).
Here $M\times N$ is the linear space of all pairs $(x,y)$ with $x\in M, y\in N$ endowed with the norm $\|(x,y)\|=\|x\|+\|y\|$.
One can easily check that $M$ is complemented in $X$ if and only if there exists a continuous linear projection onto $M$, i.e.,
a continuous linear operator $P:X\to X$ such that $Px\in M$ for all $x\in X$ and $Px=x$ for $x\in M$.

Each complemented subspace is closed (this follows from the fact that $M=S(M\times\{0\})$).
Note that one can give the following (equivalent) definition of complementability:
a subspace $M$ is said to be complemented in $X$ if $M$ is closed and there exists a closed subspace $N$ such that $M\cap N=\{0\}$ and $M+N=X$
(the equivalence of the definition to the original follows from the fact that each complemented subspace is closed and
the Banach inverse mapping theorem).

If $X$ is a Hilbert space, then each closed subspace $M$ of $X$ is complemented in $X$
(one can consider the orthogonal decomposition $X=M\oplus M^\bot$ or, equivalently, the orthogonal projection onto $M$).
Of course, this is true if $X$ is isomorphic to a Hilbert space.
But if $X$ is not isomorphic to a Hilbert space, then by the Lindenstrauss-Tzafriri theorem $X$ contains a closed subspace
which is not complemented in $X$.

For further information on complemented and uncomplemented subspaces in Banach spaces and, in particular,
various examples of uncomplemented closed subspaces see, e.g., \cite{Kadets73}, \cite{Moslehian06} and the references therein.

\subsection{}
Let $X$ be a Banach space and $X_1,...,X_n$ be complemented subspaces of $X$.
Define the sum of $X_1,...,X_n$ in the natural way, namely,
$$X_1+...+X_n:=\{x_1+...+x_n\mid x_1\in X_1,...,x_n\in X_n\}.$$
The natural question arises:

\noindent
\textbf{Question 1:}
\begin{center}
\textit{Is $X_1+...+X_n$ complemented in $X$?}
\end{center}

Note that Question 1 makes sense --- the sum of two complemented subspaces may be uncomplemented and even nonclosed.
A simple example: let $X$ be a Hilbert space, then a subspace is complemented if and only if it is closed,
and there are well-known simple examples of two closed subspaces with nonclosed sum.
Note that even if the sum of two complemented subspaces is closed, it can be uncomplemented.
An example: let $Y$ be a closed uncomplemented subspace of a Banach space $Z$;
take $X=Y\times Z$,
$$X_1=\{(y,0)\mid y\in Y\},\qquad X_2=\{(y,y)\mid y\in Y\}.$$
It is easily seen that $X_1$ and $X_2$ are complemented in $X$ but the sum $X_1+X_2=Y\times Y$ is not.

\subsection{}
If Question 1 has positive answer, then the next natural question arises:

\noindent
\textbf{Question 2:}
\begin{center}
\textit{Suppose that we know some (continuous linear) projections $P_1,...,P_n$ onto $X_1,...,X_n$, respectively.
Is there a formula for a projection onto $X_1+...+X_n$ (in terms of $P_1,...,P_n$) (of course, under certain conditions)?}
\end{center}

\subsection{}
Since each complemented subspace is closed, Question 1 is closely related to the following

\noindent
\textbf{Question 3:}
\begin{center}
\textit{Is $X_1+...+X_n$ closed in $X$?}
\end{center}

It is worth mentioning that if $X$ is a Hilbert space, then Question 1 coincides with Question 3.

Systems of subspaces $X_1,...,X_n$ for which Question 3 is very important arise in various branches of mathematics, for example, in
\begin{enumerate}
\item
theoretical tomography and theory of ridge functions (plane waves).
Here the problem on the closedness of the sum of spaces of functions, which are constant on certain sets, naturally arises.
See, e.g., \cite{SheppKruskal78}, \cite[Introduction, Chapter 7 and the references therein]{Pinkus15};
\item
theory of wavelets and multiresolution analysis.
Here the problem on the closedness of the sum of shift-invariant subspaces of $L^2(\mathbb{R}^d)$ is studied.
See, e.g., \cite{Kim06} and references therein;
\item
statistics.
See, e.g., \cite{Bickel91}, where
the closedness of the sum of two marginal subspaces is important for
constructing an efficient estimation of linear functionals of a probability measure with known marginal distributions;
\item
approximation algorithms in Hilbert and Banach spaces and, in particular, methods of alternating projections.
See, e.g., \cite[Chapter 9 and the bibliography therein]{Pinkus15}, \cite[Theorem 5.19]{BauschkeBorwein96}, \cite[Theorem 4.1]{Badea12},
\cite[Section 3]{Pustylnik12};
\item
a problem of finding an element of a Hilbert space with prescribed best approximations from a finite number of subspaces.
This problem is a common problem in applied mathematics, it arises in harmonic analysis, optics, and signal theory.
See, e.g., \cite{Combettes10} and references therein;
\item
theory of Banach algebras.
See, e.g., \cite{Rudin75, Dixon00, Dudziak00};
\item
theory of operator algebras.
See, e.g., \cite{Hartz12}, where
the closedness of finite sums of full Fock spaces over subspaces of $\mathbb{C}^d$ plays a crucial role
for construction of a topological isomorphism between universal operator algebras;
\item
quadratic programming.
See, e.g., \cite{Schochetman01};
\item
theory of $\mu$-pseudo almost periodic functions (or sequences) and $\mu$-pseudo almost automorphic functions (or sequences).
See, e.g., \cite{Blot16};
\end{enumerate}
and others.

\subsection{}
Another property of systems of subspaces which will be of interest to us is the linear independence of the subspaces.
A system of subspaces $X_1,...,X_n$ is said to be linearly independent if an equality $x_1+...+x_n=0$, where $x_1\in X_1,...,x_n\in X_n$,
implies that $x_1=...=x_n=0$.
Why we are interested in the linear independence property?
The thing is that the properties of linear independence of a system of subspaces and closedness of their sum are closely related
to the inverse best approximation property of a system of subspaces of a Hilbert space and its natural generalization,
the interpolation property of a system of operators.
Let us explain this relation.
In \cite{Combettes10} the authors study systems of closed subspaces $H_1,...,H_n$ of a Hilbert space $H$ which have the following property:
for arbitrary elements $h_1\in H_1,...,h_n\in H_n$ there exists an element $h\in H$ such that $P_{H_i}h$, the orthogonal projection of $h$ onto $H_i$,
is equal to $h_i$ for every $i=1,...,n$.
The property of a system of subspaces is called the inverse best approximation property (IBAP).
The property has the following natural generalization.
Let $V,W_1,...,W_n$ be Banach spaces and $A_i:V\to W_i$, $i=1,...,n$ be continuous linear operators.
We will say that the system of operators $A_1,...,A_n$ possesses the interpolation property (IP) if
for arbitrary elements $w_1\in W_1,...,w_n\in W_n$ there exists an element $v\in V$ such that $A_i v=w_i$ for $i=1,...,n$.
Note that if $V=H$, $W_i=H_i$ and $A_i h=P_{H_i}h$, $h\in H$ for $i=1,...,n$, then
the (IP) for $A_1,...,A_n$ coincides with the (IBAP) for $H_1,...,H_n$.
Using arguments similar to that in \cite[Subsection 4.2]{Feshchenko12}, one can show that
a system of operators $A_1,...,A_n$ possesses the (IP) if and only if the following two conditions hold:

(1) the range of $A_i$ is equal to $W_i$ for $i=1,...,n$;

(2) the subspaces $(\ker(A_1))^\bot,...,(\ker(A_n))^\bot$ are linearly independent and their sum is closed in $V^*$.
Here for a subset $M\subset V$ we denote by $M^\bot$ the annihilator of $M$, i.e., the set of all $\varphi\in V^*$ such that $\varphi|_M=0$.

In particular, a system of closed subspaces $H_1,...,H_n$ of a Hilbert space $H$ possesses the (IBAP) if and only if
$H_1,...,H_n$ are linearly independent and their sum is closed in $H$.

\subsection{}
The paper is organized as follows.
In Section \ref{s:B} we study Questions 1 and 2 in the (general) Banach space setting.
In Subsection \ref{ss:B trivial observations} we make a few simple observations on the questions.
In Subsection \ref{ss:B known results} we present known results.
Our results are presented in Subsections \ref{ss:B result 1}, \ref{ss:B result 2}, and \ref{ss:B necessity r(E)<1}.
We provide a sufficient condition for the sum of complemented subspaces of a Banach space to be complemented.
Under this condition a formula for a projection onto the sum is given (see Theorems~\ref{Th:Banach} and~\ref{Th:rate Banach}).
We also show that the condition is sharp (in a certain sense) (see Theorem~\ref{Th:Banach r(E)=1}).
Proofs of the results are given in Subsections \ref{ss:B proof 1,2} and \ref{ss:B proof 3}.

As applications of Theorem~\ref{Th:Banach} we get

(1) a sufficient condition for the complementability of sums of marginal subspaces in $L^p$
(see Section~\ref{s:marginal}, the main result is Theorem~\ref{Th:Marginal});

(2) sufficient conditions for the complementability of sums of tensor powers of subspaces in a tensor power of a Banach space
(see Section~\ref{s:tensor}, main results are Theorems~\ref{Th:Tensor} and~\ref{Th:Tensor E});

(3) quantitative results on stability of the complementability property of the sum of linearly independent subspaces
(see Section~\ref{s:stability}, main results are presented in Subsection~\ref{ss:stability complete proof}).

\subsection{Notation}
Throughout the paper, $X$ is a real or complex Banach space with norm $\|\cdot\|$.
The identity operator on $X$ is denoted by $I$ (throughout the paper it is clear which Banach space is being considered).
All operators in the paper are continuous linear operators.
In particular, by a projection we always mean a continuous linear projection.
The kernel and range of an operator $T$ will be denoted by $\ker(T)$ and $Ran(T)$, respectively.
All vectors are vector-columns; the letter "t" means transpose.

\section{On sums of complemented subspaces}\label{s:B}

Let $X$ be a Banach space, $X_1,...,X_n$ be complemented subspaces of $X$ and $P_1,...,P_n$ be projections onto $X_1,...,X_n$, respectively.

\subsection{Simple observations}\label{ss:B trivial observations}

We begin with a few simple observations on Questions 1 and 2.
These observations were used by many authors.

(1)
If $P_i|_{X_j}=0$ for all $i\neq j$, $i,j\in\{1,...,n\}$, then $X_1,...,X_n$ are linearly independent, their sum is complemented in $X$ and
$$P=P_1+...+P_n$$
is a projection onto $X_1+...+X_n$.

\begin{remark}\label{r:converse}
The ``converse'' is also true.
More precisely, let $V_1,...,V_n$ be closed subspaces of $X$.
If $V_1,...,V_n$ are linearly independent and their sum is complemented in $X$, then there \emph{exist}
projections $Q_1,...,Q_n$ onto $V_1,...,V_n$, respectively, such that $Q_i|_{V_j}=0$ for all $i\neq j$.
Let us prove this.
Denote by $V$ a complement of $V_1+...+V_n$ in $X$.
Then $V$ is closed, the subspaces $V_1,...,V_n,V$ are linearly independent and their sum is equal to $X$.
Let $V_1\times...\times V_n\times V$ be the linear space of all vector-columns $(v_1,...,v_n,v)^t$ with $v_1\in V_1,...,v_n\in V_n,v\in V$
endowed with the norm $\|(v_1,...,v_n,v)^t\|=\|v_1\|+...+\|v_n\|+\|v\|$.
Then, obviously, $V_1\times...\times V_n\times V$ is a Banach space.
Define the sum operator $S:V_1\times...\times V_n\times V\to X$ by
$$S(v_1,...,v_n,v)^t=v_1+...+v_n+v,\quad v_1\in V_1,...,v_n\in V_n,v\in V.$$
Then $\ker(S)=\{0\}$ and $Ran(S)=X$.
It follows that $S^{-1}$ is bounded.
Denote by $\pi_i$ the natural projection $\pi_i:V_1\times...\times V_n\times V\to V_i$ and set $Q_i=\pi_i S^{-1}$, that is,
$Q_i x=v_i$ if $x=v_1+...+v_n+v$, where $v_1\in V_1,...,v_n\in V_n,v\in V$.
It is clear that $Q_i$ is a bounded projection onto $V_i$ and $Q_i|_{V_j}=0$ for all $j\neq i$.
\end{remark}

(2)
Let $n=2$.
If $P_2|_{X_1}=0$, that is, $P_2 P_1=0$, then the subspaces $X_1,X_2$ are linearly independent (i.e., $X_1\cap X_2=\{0\}$),
their sum is complemented in $X$ and
$$P=P_1+P_2-P_1 P_2$$
is a projection onto $X_1+X_2$.
Note that $P_1+P_2-P_1 P_2=I-(I-P_1)(I-P_2)$.
Now an induction argument shows that if $P_i|_{X_j}=0$ for all $i>j$, $i,j\in\{1,...,n\}$,
then $X_1,...,X_n$ are linearly independent, their sum is complemented in $X$ and
$$P=I-(I-P_1)(I-P_2)...(I-P_n)$$
is a projection onto $X_1+...+X_n$.

(3)
(see, e.g., \cite[Lemma 2.6]{Svensson87})
Let $n=2$.
If $X_2$ is finite dimensional, then $X_1+X_2$ is complemented in $X$.

Indeed, we can assume that $X_2\cap X_1=\{0\}$.
Using the Hahn-Banach theorem one can easily construct a projection $P_2$ onto $X_2$ such that $P_2|_{X_1}=0$.
Now observation (2) shows that $X_1+X_2$ is complemented in $X$.

\subsection{Known results}\label{ss:B known results}

Questions 1 and 2 seem to be very basic in the theory of complemented subspaces,
but, to our knowledge, there are only a few known results (in the general Banach space setting).
Let us present them.

For $n=2$ each of the following conditions is sufficient for $X_1+X_2$ to be complemented in $X$:
\begin{enumerate}
\item
(Alan LaVergne, 1979, \cite[Proposition]{LaVergne79})
$P_2 P_1$ is strictly singular.
In fact, the proof given in \cite{LaVergne79} works for the case when $I-P_2 P_1$ is Fredholm of index zero;
\item
(Lars Svensson, 1987, \cite[Lemma 2.5]{Svensson87})
$\ker(I-P_2 P_1)=\ker(I-P_1 P_2)=X_1\cap X_2$ is complemented in $X$ and
$Ran(I-P_2 P_1)$, $Ran(I-P_1 P_2)$ are also complemented in $X$;
\item
(\cite[Theorem 2.8]{Svensson87})
$I-P_2 P_1$ and $I-P_1 P_2$ are Fredholm of index zero.
In fact, the proof given in \cite{Svensson87} works for the case when $I-P_2 P_1$ and $I-P_1 P_2$ are Fredholm;
\item
(Manuel Gonzalez, 1994, \cite[Lemma 1]{Gonzalez94})
$P_2 P_1$ is inessential.
The proof given in \cite{Gonzalez94} repeats that of \cite{LaVergne79}
(note that if an operator $A:X\to X$ is inessential, then $I-A$ is Fredholm of index zero).
\item
(S\"{u}leyman \"{O}nal and Murat Yurdakul, 2013, \cite{Yurdakul})
the restriction of the operator $I-P_2 P_1$ to its invariant subspace $X_2$ is Fredholm.
One can easily check that the condition is equivalent to the following: the operator $I-P_2 P_1$ is Fredholm.

We should note that \cite[Proofs of Lemma 1 and Proposition 2]{Yurdakul} show more.
Namely, if $\ker(I-P_2 P_1)$ is finite dimensional and $(I-P_2 P_1)(X_2)$ is complemented in $X_2$,
then $X_1+X_2$ is complemented in $X$.
\end{enumerate}

Concerning Question 2, a few formulas for a projection onto $X_1+X_2$ (under certain conditions) can be found in \cite{Svensson87}.
For example, if $\ker(I-P_2 P_1)=\ker(I-P_1 P_2)=\{0\}$ and $Ran(I-P_2 P_1)$, $Ran(I-P_1 P_2)$ are complemented in $X$, then
$$P=P_1 A_{21}(I-P_2)+P_2 A_{12}(I-P_1)$$
is a projection onto $X_1+X_2$, here $A_{12}$ and $A_{21}$ are left-inverses for $I-P_1 P_2$ and $I-P_2 P_1$, respectively.
One more formula can be obtained by \cite[Proofs of Lemma 1 and Proposition 2]{Yurdakul}.

For arbitrary $n$ each of the following conditions is sufficient for $X_1+...+X_n$ to be complemented in $X$:
\begin{enumerate}
\item
(\cite[Corollary]{LaVergne79})
$X_1,...,X_n$ are pairwise totally incomparable.
We should note that using LaVergne's proof of \cite[Proposition]{LaVergne79} one can get a stronger result.
In fact, using the proof one can easily show that if $P_2 P_1$ is strictly singular,
then there exists a projection $P$ onto $X_1+X_2$ such that $P$ equals $P_1+P_2-P_1 P_2$ modulo strictly singular operators.
Now an induction argument shows that if $P_i P_j$ is strictly singular for each pair $i>j$, $i,j\in\{1,...,n\}$,
then $X_1+...+X_n$ is complemented in $X$ and
there exists a projection $P$ onto $X_1+...+X_n$ such that $P$ equals
$$I-(I-P_1)...(I-P_n)$$
modulo strictly singular operators.
\item
(\cite[Corollary 2.9]{Svensson87})
$P_i P_j$ is compact for every pair $i\neq j$, $i,j\in\{1,...,n\}$.
Moreover, under this condition there exists a projection $P$ onto $X_1+...+X_n$ such that $P$ equals
$$P_1+...+P_n$$
modulo compact operators.
\end{enumerate}

\subsection{Our result}\label{ss:B result 1}

In this subsection we provide a sufficient condition for $X_1+...+X_n$ to be complemented in $X$.
Under the condition a formula for a projection onto the sum is given.
The result can be regarded as a strengthening of observation (1) in Subsection \ref{ss:B trivial observations}.

Suppose that nonnegative numbers $\varepsilon_{ij}$, $i\neq j$, $i,j\in\{1,...,n\}$ are such that
\begin{equation}\label{eq:varepsilon_ij}
\|P_i x\|\leqslant\varepsilon_{ij}\|x\|,\quad x\in X_j
\end{equation}
for every $i\neq j$, $i,j\in\{1,...,n\}$.

\begin{remark}
It is clear that \eqref{eq:varepsilon_ij} is equivalent to the inequality $\|P_i|_{X_j}\|\leqslant\varepsilon_{ij}$.
The reader may wonder why we don't set $\varepsilon_{ij}:=\|P_i|_{X_j}\|$.
Answer: we believe that \eqref{eq:varepsilon_ij} is more convenient for applications.
Indeed, finding the exact value of $\|P_i|_{X_j}\|$ is usually much more difficult than obtaining an inequality of the form \eqref{eq:varepsilon_ij}.
\end{remark}

Define the $n\times n$ matrix $E=(e_{ij})$ by
$$e_{ij}=
\begin{cases}
0, &\text{if $i=j$;}\\
\varepsilon_{ij}, &\text{if $i\neq j$.}
\end{cases}
$$
Denote by $r(E)$ the spectral radius of $E$.
Set $A:=P_1+...+P_n$.

Now we are ready to formulate our first result.

\begin{theorem}\label{Th:Banach}
If $r(E)<1$, then the subspaces $X_1,...,X_n$ are linearly independent, their sum is complemented in $X$ and the subspace $\ker(P_1)\cap...\cap\ker(P_n)$
is a complement of $X_1+...+X_n$ in $X$.
Moreover, the sequence of operators
$$I-(I-A)^N$$
converges uniformly to the projection $P$ onto $X_1+...+X_n$ along $\ker(P_1)\cap...\cap\ker(P_n)$ as $N\to\infty$.
\end{theorem}

\begin{remark}
Theorem~\ref{Th:Banach} provides a sufficient condition for $n$ subspaces to be linearly independent and their sum to be complemented.
Even for $n=2$ this condition is not necessary.
To see this, consider the following simple example.
Let $X=\mathbb{R}^2$ with the Euclidean norm.
Set $v_1=(1,0)^t$, $v_2=(0,1)^t$ and let $X_i$ be the subspace spanned by $v_i$, $i=1,2$.
For two real numbers $a,b$ define projections $P_1,P_2$ onto $X_1,X_2$ by
$$
P_1=\begin{pmatrix}
1 & a\\
0 & 0
\end{pmatrix},
P_2=\begin{pmatrix}
0 & 0\\
b & 1
\end{pmatrix}.
$$
Then $P_1 v_2=av_1$ and therefore $\|P_1|_{X_2}\|=|a|$.
Also, $P_2 v_1=bv_2$ and therefore $\|P_2|_{X_1}\|=|b|$.
So for the optimal choice $\varepsilon_{12}=|a|$ and $\varepsilon_{21}=|b|$ we have
$$
E=\begin{pmatrix}
0  & |a|\\
|b|& 0
\end{pmatrix}
$$
and $r(E)=\sqrt{|ab|}$ can be arbitrarily large.
However, the subspaces $X_1,X_2$ are linearly independent and their sum $X_1+X_2=X$ is complemented in $X$.
\end{remark}

\begin{remark}
For $n=2$ the inequality $r(E)<1$ is equivalent to
$$\varepsilon_{12}\varepsilon_{21}<1.$$
For $n=3$ the inequality $r(E)<1$ is equivalent to
$$\varepsilon_{12}\varepsilon_{21}+\varepsilon_{23}\varepsilon_{32}+\varepsilon_{31}\varepsilon_{13}+
\varepsilon_{12}\varepsilon_{23}\varepsilon_{31}+\varepsilon_{21}\varepsilon_{32}\varepsilon_{13}<1.$$
For arbitrary $n\geqslant 2$, $r(E)<1$ if and only if each principal minor of the matrix $I-E$ is positive.
(Recall that a principal minor is the determinant of a principal submatrix;
a principal submatrix is a square submatrix obtained by removing certain rows and columns with the same index sets.)
This fact is an easy consequence of the theory of nonnegative matrices (see, e.g., \cite[Chapter 8]{HornJohnson13}).
\end{remark}

\subsection{A rate of convergence}\label{ss:B result 2}

For practical applications it is important to know how fast does the sequence $I-(I-A)^N$ converge to $P$.
Our next result shows that the rate of convergence can be estimated from above by $C\alpha^N$, where $\alpha\in[0,1)$.
To formulate the result we need the following notation: for two vectors $u,v\in\mathbb{R}^n$ we will write $u\leqslant v$ if $u\leqslant v$ coordinatewise.

\begin{theorem}\label{Th:rate Banach}
The following statements on the rate of convergence of $I-(I-A)^N$ to $P$ are true.
\begin{enumerate}
\item
Suppose a vector $w=(w_1,...,w_n)^t$ with positive coordinates and a number $\alpha\in[0,1)$ satisfy $Ew\leqslant\alpha w$.
Then
$$\|I-(I-A)^N-P\|\leqslant(w_1+...+w_n)\max\{(1/w_1)\|P_1\|,...,(1/w_n)\|P_n\|\}\frac{\alpha^N}{1-\alpha}$$
for each $N\geqslant 1$.
\item
Suppose a vector $w=(w_1,...,w_n)^t$ with positive coordinates and a number $\alpha\in[0,1)$ satisfy $E^t w\leqslant \alpha w$.
Then
$$\|I-(I-A)^N-P\|\leqslant(w_1\|P_1\|+...+w_n\|P_n\|)\max\{(1/w_1),...,(1/w_n)\}\frac{\alpha^N}{1-\alpha}$$
for each $N\geqslant 1$.
\end{enumerate}
\end{theorem}

\begin{remark}
Since $E$ is a nonnegative matrix,
the existence of a vector $w\in\mathbb{R}^n$ with positive coordinates and a number $\alpha\in[0,1)$ such that $Ew\leqslant\alpha w$
is equivalent to $r(E)<1$.
More precisely, if such $w$ and $\alpha$ exist, then $r(E)\leqslant\alpha<1$ (see \cite[Corollary 8.1.29]{HornJohnson13}).
Conversely, suppose that $r(E)<1$.
If $E$ is irreducible, then one can take $\alpha$ to be $r(E)$ and $w$ a Perron-Frobenius vector of $E$.
If $E$ is not irreducible, then consider the matrix $E'=(e_{ij}+\delta)$ for sufficiently small $\delta>0$,
and take $\alpha$ to be $r(E')$ and $w$ a Perron-Frobenius vector of $E'$.

Similarly,
the existence of a vector $w$ with positive coordinates and a number $\alpha\in[0,1)$ such that $E^t w\leqslant \alpha w$
is equivalent to $r(E)<1$.
\end{remark}

Using Theorem~\ref{Th:rate Banach}, we can get concrete estimates for the rate of convergence of $I-(I-A)^N$ to $P$.
Suppose $E$ is irreducible and $r(E)<1$.
Take $\alpha$ to be $r(E)$ and $w$ a Perron-Frobenius vector of $E$.
Then we get
$$\|I-(I-A)^N-P\|\leqslant(w_1+...+w_n)\max\{(1/w_1)\|P_1\|,...,(1/w_n)\|P_n\|\}\frac{(r(E))^N}{1-r(E)}.$$
Similarly, we can take $\alpha$ to be $r(E)$ and $w$ a Perron-Frobenius vector of $E^t$.
Then we get
$$\|I-(I-A)^N-P\|\leqslant(w_1\|P_1\|+...+w_n\|P_n\|)\max\{(1/w_1),...,(1/w_n)\}\frac{(r(E))^N}{1-r(E)}.$$

\begin{remark}
In the study of Questions 1 and 2 one can assume that $E$ is irreducible.
Indeed, suppose that $E$ is reducible and $r(E)<1$.
Then, up to a permutation of the subspaces $X_1,...,X_n$, the matrix $E$ has the form
$$E=\begin{pmatrix}
E_1   &*     & ...  &*\\
0     &E_2   &\ddots&\vdots\\
\vdots&\ddots&\ddots&*\\
0     &...   &0     &E_m
\end{pmatrix},$$
where $E_1,...,E_m$ are irreducible and $r(E_i)<1$ for $i=1,...,m$.
Now we apply Theorem~\ref{Th:Banach} to the first group of subspaces (i.e. $X_1,...,X_{n_1}$, where $n_1$ is order of the matrix $E_1$)
with the corresponding matrix $E_1$.
Then we see that their sum $\widetilde{X}_1$ is complemented in $X$
and $I-(I-A_1)^N$ converges to a projection $\widetilde{P}_1$ onto $\widetilde{X}_1$ as $N\to\infty$.
Similarly, we apply Theorem~\ref{Th:Banach} to each of the remaining $m-1$ groups of subspaces.
Then we see that $\widetilde{X}_i$, the sum of subspaces of the $i$-th group, is complemented in $X$ and
$I-(I-A_i)^N$ converges to a projection $\widetilde{P}_i$ onto $\widetilde{X}_i$ as $N\to\infty$, $i=1,...,m$.
Clearly, $\widetilde{P}_i|_{\widetilde{X}_j}=0$ for every pair $i>j$.
Now observation (2) in Subsection \ref{ss:B trivial observations} shows that $\widetilde{X}_1+...+\widetilde{X}_m=X_1+...+X_n$ is complemented in $X$ and
$$I-(I-\widetilde{P}_1)...(I-\widetilde{P}_m)$$
is a projection onto $X_1+...+X_n$.
\end{remark}

\subsection{On the necessity of the condition $r(E)<1$}\label{ss:B necessity r(E)<1}

The assumption $r(E)<1$ is a \textit{sharp} sufficient condition for $X_1+...+X_n$ to be complemented in $X$.
More precisely, we have the following result.

\begin{theorem}\label{Th:Banach r(E)=1}
Let $E=(e_{ij})$ be an $n\times n$ matrix with $e_{ii}=0$ for $i=1,...,n$ and $e_{ij}\geqslant 0$ for every pair $i\neq j$, $i,j\in\{1,...,n\}$.
If $r(E)=1$, then there exist a Banach space $X$, complemented subspaces $X_1,...,X_n$ of $X$ and projections $P_1,...,P_n$ onto $X_1,...,X_n$, respectively, such that
\begin{enumerate}
\item
$\|P_i x\|=e_{ij}\|x\|$, $x\in X_j$, for each pair $i\neq j$, $i,j\in\{1,...,n\}$;
\item
$X_1,...,X_n$ are linearly independent;
\item
$X_1+...+X_n$ is closed and not complemented in $X$.
\end{enumerate}
\end{theorem}

\begin{remark}
In the case when $r(E)>1$ the theorem can be applied to the matrix $(1/r(E))E$.
\end{remark}

\subsection{Proof of Theorems~\ref{Th:Banach} and~\ref{Th:rate Banach}}\label{ss:B proof 1,2}

First, we will prove Theorem~\ref{Th:Banach} and the first part of Theorem~\ref{Th:rate Banach}.
Thus we assume that a vector $w=(w_1,...,w_n)^t$ with positive coordinates and a number $\alpha\in[0,1)$ satisfy $Ew\leqslant\alpha w$.

Let $X_1\times...\times X_n$ be the linear space of all vector-columns $(x_1,...,x_n)^t$ with $x_1\in X_1,...,x_n\in X_n$
endowed with the weighted $\infty$-norm
$$\|(x_1,...,x_n)^t\|=\max\{(1/w_1)\|x_1\|,...,(1/w_n)\|x_n\|\}.$$
Then, obviously, $X_1\times...\times X_n$ is a Banach space.
Define the operator $S:X_1\times...\times X_n\to X$ by
$$S(x_1,...,x_n)^t=x_1+...+x_n,\quad x_1\in X_1,...,x_n\in X_n,$$
and the operator $J:X\to X_1\times...\times X_n$ by
$$Jx=(P_1 x,...,P_n x)^t,\quad x\in X.$$
Then $SJ=P_1+...+P_n=A$.
Set
$$G=JS:X_1\times...\times X_n\to X_1\times...\times X_n.$$
Let $(G_{ij}:X_j\to X_i\mid i,j=1,...,n)$ be the block decomposition of $G$.
It is clear that $G_{ij}$ acts as $P_i$ on $X_j$.
In particular, $G_{ii}=I$ for $i=1,...,n$.

Let us show that $G$ is invertible.
To this end we will estimate $\|G-I\|$.
For the block decomposition of $G-I$ we have $(G-I)_{ii}=0$ for $i=1,...,n$, and $(G-I)_{ij}=G_{ij}$ for $i\neq j$.
Then $\|(G-I)_{ij}\|\leqslant\varepsilon_{ij}$ for $i\neq j$ and thus $\|(G-I)_{ij}\|\leqslant e_{ij}$ for every pair $i,j$.
It follows easily that $\|G-I\|\leqslant\|E\|$, where $\|E\|$ is the operator norm of the matrix $E$ considered as the operator
on the space $\mathbb{R}^n$ endowed with the weighted $\infty$-norm $\|u\|=\max\{(1/w_1)|u_1|,...,(1/w_n)|u_n|\}$, $u=(u_1,...,u_n)^t\in\mathbb{R}^n$.
But
$$\|E\|=\max\left\{(e_{i1}w_1+e_{i2}w_2+...+e_{in}w_n)/w_i\mid i=1,...,n\right\}\leqslant\alpha.$$
Therefore $\|G-I\|\leqslant\alpha<1$.
Consequently, $G$ is invertible and
\begin{equation}\label{eq:G^-1}
G^{-1}=(I-(I-G))^{-1}=\sum_{k=0}^\infty (I-G)^k,
\end{equation}
where the series converges uniformly.

Since $G$ is invertible, we conclude that $\ker(G)=\{0\}$.
Thus $\ker(S)=\{0\}$.
It follows that $X_1,...,X_n$ are linearly independent.

Now we claim that $P=SG^{-1}J:X\to X$ is a projection onto $X_1+...+X_n$.
Indeed, the operator $P$ has the following properties:
\begin{enumerate}
\item
$P$ is a continuous linear operator;
\item
$Ran(P)\subset Ran(S)=X_1+...+X_n$;
\item
$Px=x$ for every $x\in X_1+...+X_n$.
Indeed, $x=Sv$ for some $v\in X_1\times...\times X_n$.
Then
$$Px=SG^{-1}JSv=SG^{-1}Gv=Sv=x.$$
\end{enumerate}
These three properties of $P$ imply that $P$ is a projection onto $X_1+...+X_n$.
Hence $X_1+...+X_n$ is complemented in $X$.

Further, $\ker(P)$ is a complement of $X_1+...+X_n$ in $X$.
It is easily seen that
$$\ker(P)=\ker(J)=\ker(P_1)\cap...\cap\ker(P_n).$$
Hence, $\ker(P_1)\cap...\cap\ker(P_n)$ is a complement of $X_1+...+X_n$ in $X$ and
$P$ is the projection onto $X_1+...+X_n$ along $\ker(P_1)\cap...\cap\ker(P_n)$.

Let us show that the sequence of operators $I-(I-A)^N$ converges uniformly to $P$ as $N\to\infty$.
Using~\eqref{eq:G^-1} we get
$$P=S\left(\sum_{k=0}^\infty (I-G)^k\right)J=\lim_{N\to\infty}S\left(\sum_{k=0}^{N-1}(I-G)^k\right)J.$$
Since
$$S(I-G)=S(I-JS)=(I-SJ)S=(I-A)S$$
we see that $S(I-G)^k=(I-A)^k S$ for $k=0,1,...$.
Therefore
\begin{align*}
P&=\lim_{N\to\infty}\left(\sum_{k=0}^{N-1}(I-A)^k S\right)J=\lim_{N\to\infty}\left(\sum_{k=0}^{N-1}(I-A)^k\right)A=\\
&=\lim_{N\to\infty}\left(\sum_{k=0}^{N-1}(I-A)^k\right)(I-(I-A))=\lim_{N\to\infty}(I-(I-A)^N).
\end{align*}
This finishes the proof of Theorem~\ref{Th:Banach}.

It remains to estimate $\|I-(I-A)^N-P\|$.
We have
$$\|I-(I-A)^N-P\|=\|S\left(\sum_{k=N}^\infty (I-G)^k\right)J\|\leqslant \|S\|\|J\|\sum_{k=N}^\infty\|I-G\|^k.$$
From the definitions of $S$ and $J$ we have
$$\|S\|\leqslant w_1+...+w_n$$
and
$$\|J\|=\max\{(1/w_1)\|P_1\|,...,(1/w_n)\|P_n\|\}.$$
Also, recall that $\|G-I\|\leqslant\alpha$.
Therefore
$$\|I-(I-A)^N-P\|\leqslant(w_1+...+w_n)\max\{(1/w_1)\|P_1\|,...,(1/w_n)\|P_n\|\}\frac{\alpha^N}{1-\alpha}.$$
This finishes the proof of the first part of Theorem~\ref{Th:rate Banach}.

The proof of the second part of Theorem~\ref{Th:rate Banach} follows the same lines as the one for the first part
but with the only difference: instead of the weighted $\infty$-norm on the linear space $X_1\times...\times X_n$
one should consider the weighted 1-norm
$$\|(x_1,...,x_n)^t\|=w_1\|x_1\|+...+w_n\|x_n\|.$$

\subsection{Proof of Theorem~\ref{Th:Banach r(E)=1}}\label{ss:B proof 3}

Our construction of a space $X$, its subspaces $X_1,...,X_n$ and projections $P_1,...,P_n$ is based on the following simple observation.
Let $\langle\cdot,\cdot\rangle$ be the standard inner product in $\mathbb{R}^n$, i.e.,
$$\langle u,v\rangle=u_1 v_1+...+u_n v_n,$$
where $u=(u_1,...,u_n)^t$ and $v=(v_1,...,v_n)^t$.
Each nonzero vector $v\in\mathbb{R}^n$ spans the one-dimensional subspace
$$L_{\mathbb{R}}(v)=\{\lambda v\mid\lambda\in\mathbb{R}\}=\{(\lambda v_1,...,\lambda v_n)^t\mid\lambda\in\mathbb{R}\}.$$
If a vector $u\in\mathbb{R}^n$ satisfies $\langle v,u\rangle=1$, then the mapping
$$x\mapsto \langle x,u\rangle v,\quad x\in\mathbb{R}^n$$
is a projection onto $L_{\mathbb{R}}(v)$.

To construct a space $X$, its subspaces $X_1,...,X_n$ and projections $P_1,...,P_n$ we need
two collections of vectors $u^{(i)}\in\mathbb{R}^n$, $i=1,...,n$ and $v^{(j)}\in\mathbb{R}^n$, $j=1,...,n$
which have the following properties:
\begin{enumerate}
\item
$v^{(1)},...,v^{(n)}$ are unit basis vectors of $\mathbb{R}^n$;
\item
$\langle v^{(i)},u^{(i)}\rangle=1$ for $i=1,...,n$ and
$|\langle v^{(j)},u^{(i)}\rangle|=e_{ij}$ for each pair $i\neq j$, $i,j\in\{1,...,n\}$;
\item
the $n$-th coordinate of the vectors $u^{(1)},...,u^{(n)}$ equals $0$.
\end{enumerate}
Such vectors can be constructed as follows.
Let $f^{(i)}\in\mathbb{R}^n$ be the transpose of the $i$-th row of the matrix $I-E$, i.e.,
$$f^{(i)}=(-e_{i1},...,-e_{i,i-1},1,-e_{i,i+1},...,-e_{i,n})^t,\qquad i=1,...,n.$$
Denote by $g^{(j)}$, $j=1,...,n$, the standard unit basis vectors of $\mathbb{R}^n$, i.e.,
$$g^{(j)}=(0,...,0,1,0,...,0)^t,\qquad j=1,...,n,$$
where the $1$ is in the $j$-th position.
Clearly, $\langle g^{(i)},f^{(i)}\rangle=1$ for $i=1,...,n$ and
$\langle g^{(j)},f^{(i)}\rangle=-e_{ij}$ for each pair $i\neq j$, $i,j\in\{1,...,n\}$.
Further, since $E$ is a nonnegative matrix, we conclude that $r(E)=1$ is an eigenvalue of $E$.
This means that the matrix $I-E$ is singular, i.e., the vectors $f^{(1)},...,f^{(n)}$ are linearly dependent.
It follows that the dimension of the linear span of $f^{(1)},...,f^{(n)}$ is not greater than $n-1$.
Thus there exists a unitary operator $T:\mathbb{R}^n\to\mathbb{R}^n$ such that
$T(\text{linear span of } f^{(1)},...,f^{(n)})$ is contained in the hyperplane
$\{u=(u_1,...,u_n)^t\in\mathbb{R}^n\mid u_n=0\}$.
Set $u^{(i)}=Tf^{(i)}$, $i=1,...,n$ and $v^{(j)}=Tg^{(j)}$, $j=1,...,n$.
It is clear that these two collections of vectors have the required properties.

Now we are ready to construct a space $X$, its subspaces $X_1,...,X_n$ and projections $P_1,...,P_n$.
Let $Y$ be a closed uncomplemented subspace of a Banach space $Z$.
Define $X$ to be the linear space
$$\underbrace{Y\times...\times Y}_{n-1}\times Z$$
of all vector-columns $x=(y_1,...,y_{n-1},z)^t$ with $y_1\in Y,...,y_{n-1}\in Y,z\in Z$ endowed with the norm
$$\|x\|=(\|y_1\|^2+...+\|y_{n-1}\|^2+\|z\|^2)^{1/2}.$$
Then, obviously, $X$ is a Banach space.

To make our construction of subspaces $X_1,...,X_n$ and projections $P_1,...,P_n$ more clear we introduce the following notation.
For $y\in Y$ and $v=(v_1,...,v_n)^t\in\mathbb{R}^n$ we set
$$yv:=(v_1 y,...,v_n y)^t.$$
For $x=(y_1,...,y_{n-1},z)\in X$ and $u=(u_1,...,u_n)^t\in\mathbb{R}^n$ set
$$\langle x,u\rangle=u_1 y_1+...+u_{n-1}y_{n-1}+u_n z.$$
Now for each $i=1,...,n$ we define the subspace $X_i$ of $X$ by
$$X_i=L_{Y}(v^{(i)})=\{yv^{(i)}\mid y\in Y\}=\{(v^{(i)}_1 y,...,v^{(i)}_n y)^t\mid y\in Y\}$$
and the projection $P_i:X\to X$ onto $X_i$ by
\begin{align*}
P_i x&=\langle x,u^{(i)}\rangle v^{(i)}=\\
&=(v^{(i)}_1(u^{(i)}_1 y_1+...+u^{(i)}_{n-1}y_{n-1}+u^{(i)}_n z),...,v^{(i)}_n(u^{(i)}_1 y_1+...+u^{(i)}_{n-1}y_{n-1}+u^{(i)}_n z))=\\
&=(v^{(i)}_1(u^{(i)}_1 y_1+...+u^{(i)}_{n-1}y_{n-1}),...,v^{(i)}_n(u^{(i)}_1 y_1+...+u^{(i)}_{n-1}y_{n-1})).
\end{align*}

Let us show that $\|P_i x\|=e_{ij}\|x\|$, $x\in X_j$, for each pair $i\neq j$, $i,j\in\{1,...,n\}$.
Consider arbitrary $x\in X_j$.
Then $x=yv^{(j)}$ for some $y\in Y$.
Since $v^{(j)}$ is a unit vector, we see that $\|x\|=\|y\|$.
We have
$$P_i x=\langle yv^{(j)},u^{(i)}\rangle v^{(i)}=(\langle v^{(j)},u^{(i)}\rangle y)v^{(i)}.$$
Therefore
$$\|P_i x\|=\|\langle v^{(j)},u^{(i)}\rangle y\|=|\langle v^{(j)},u^{(i)}\rangle|\|y\|=e_{ij}\|x\|.$$

Since $v^{(1)},...,v^{(n)}$ are linearly independent, we conclude that $X_1,...,X_n$ are linearly independent and $X_1+...+X_n=Y\times...\times Y$.
Thus $X_1+...+X_n$ is closed in $X$.
Recall that $Y$ is not complemented in $Z$;
it follows that $X_1+...+X_n$ is not complemented in $X=Y\times...\times Y\times Z$.

\section{Sums of marginal subspaces}\label{s:marginal}

\subsection{Definitions}

Let $(\Omega,\mathcal{F},\mu)$ be a probability space.
Denote by $\mathbb{K}$ a base field of scalars, i.e., $\mathbb{R}$ or $\mathbb{C}$.
For an $\mathcal{F}$-measurable function (random variable) $\xi:\Omega\to\mathbb{K}$ denote by $E\xi$ the expectation of $\xi$ (if it exists).
Two random variables $\xi$ and $\eta$ are said to be equivalent if $\xi(\omega)=\eta(\omega)$ for $\mu$-almost all $\omega$.
For $p\in [1,\infty)\cup\{\infty\}$ denote by $L^p(\mathcal{F})=L^p(\Omega,\mathcal{F},\mu)$ the set of equivalence classes of
random variables $\xi:\Omega\to\mathbb{K}$ such that $E|\xi|^p<\infty$ if $p\in[1,\infty)$ and $\xi$ is $\mu$-essentially bounded if $p=\infty$.
For $\xi\in L^p(\mathcal{F})$ set $\|\xi\|_p=(E|\xi|^p)^{1/p}$ if $p\in[1,\infty)$ and $\|\xi\|_\infty=\text{ess sup}|\xi|$ if $p=\infty$.
Then $L^p(\mathcal{F})$ is a Banach space.
For every sub-$\sigma$-algebra $\mathcal{A}$ of $\mathcal{F}$ we define the marginal subspace corresponding to $\mathcal{A}$,
$L^p(\mathcal{A})$, as follows.
$L^p(\mathcal{A})$ consists of elements (equivalence classes) of $L^p(\mathcal{F})$ which contain at least one $\mathcal{A}$-measurable random variable.
It is clear that $L^p(\mathcal{A})$ is a complemented subspace in $L^p(\mathcal{F})$
(the conditional expectation operator $\xi\mapsto E(\xi|\mathcal{A})$ is a norm one projection onto $L^p(\mathcal{A})$).
Denote by $L^p_0(\mathcal{A})$ the subspace of all $\xi\in L^p(\mathcal{A})$ with $E\xi=0$.
This subspace is also complemented in $L^p(\mathcal{F})$
(the centered conditional expectation operator $\xi\mapsto E(\xi|\mathcal{A})-E\xi$ is a projection onto $L^p_0(\mathcal{A})$).

\subsection{Formulation of the problem}

In this section we study the following problem.
Let $\mathcal{F}_1,...,\mathcal{F}_n$ be sub-$\sigma$-algebras of $\mathcal{F}$.
Question: when is the sum of the corresponding marginal subspaces, $L^p(\mathcal{F}_1)+...+L^p(\mathcal{F}_n)$,
complemented in $L^p(\mathcal{F})$?
Since $L^p(\mathcal{F}_i)=L^p_0(\mathcal{F}_i)+\langle 1\rangle$
(here $\langle 1\rangle$ is the subspace spanned by $1$, i.e., the subspace of constant random variables),
we see that $L^p(\mathcal{F}_1)+...+L^p(\mathcal{F}_n)=L^p_0(\mathcal{F}_1)+...+L^p_0(\mathcal{F}_n)+\langle 1\rangle$.
It follows easily that $L^p(\mathcal{F}_1)+...+L^p(\mathcal{F}_n)$ is complemented in $L^p(\mathcal{F})$ if and only if
$L^p_0(\mathcal{F}_1)+...+L^p_0(\mathcal{F}_n)$ is.

Since each complemented subspace is closed, the question on complementability of the sum of marginal subspaces
is closely related to the question on closedness of the sum (and for $p=2$ these questions coincide).
Of course, $L^p(\mathcal{F}_1)+...+L^p(\mathcal{F}_n)$ is closed in $L^p(\mathcal{F})$ if and only if $L^p_0(\mathcal{F}_1)+...+L^p_0(\mathcal{F}_n)$ is.

The question on closedness of the sum of marginal subspaces arises, for example, in

(1) additive modeling.
Here each sub-$\sigma$-algebra $\mathcal{F}_i=\sigma a(\xi_i)$, the $\sigma$-algebra generated by a random variable $\xi_i$.
Consequently, each marginal subspace $L^p(\mathcal{F}_i)$
consists of (equivalence classes of) Borel measurable transformations of $\xi_i$, $f(\xi_i)$, which belong to $L^p(\mathcal{F})$.
As Andreas Buja writes in \cite[Subsection 8.1]{Buja96},
the question on closedness of $L^2_0(\mathcal{F}_1)+...+L^2_0(\mathcal{F}_n)$ is a technicality
that is at the heart of all additive modeling, including ACE (alternating conditional expectations method),
GAMs (generalized additive models) and PPR (projection pursuit regression).

(2) theory of ridge functions. See, e.g., \cite[Chapter 7]{Pinkus15}.
Note that every subspace of ridge functions $L^p(a;K)$ can be considered as marginal.

The question on closedness is not trivial; examples when $L^p(\mathcal{F}_1)+L^p(\mathcal{F}_2)$ is not closed in $L^p(\mathcal{F})$
can be found in \cite[Proposition 4.4(a)]{Ruschendorf98} (for $p\in [1,\infty)$),
\cite[Subsection 8.3]{Buja96} (for $p=2$), \cite[Section 7.2]{Pinkus15} (for $p\in [1,\infty)\cup \{\infty\}$).

Even for simple and natural $(\Omega,\mathcal{F},\mu)$ and $\mathcal{F}_1,\mathcal{F}_2$ the questions on closedness and complementability of
$L^p(\mathcal{F}_1)+L^p(\mathcal{F}_2)$ in $L^p(\mathcal{F})$ can be nontrivial.
As an example, we formulate the following problem.
Let $\Omega=\mathbb{N}\times\mathbb{N}$ and $\mathcal{F}=2^\Omega$.
Then a probability measure $\mu$ is defined by a set of numbers $\mu(\{(i,j)\})=p_{ij}\geqslant 0$, $i,j=1,2,...$ with $\sum_{i,j=1}^\infty p_{ij}=1$.
For simplicity, we assume that all $p_{ij}>0$.
Then the space $L^p(\mathcal{F})$ consists of the functions $f:\mathbb{N}\times\mathbb{N}\to\mathbb{K}$ for which
$\sum_{i,j=1}^\infty|f(i,j)|^p p_{ij}<+\infty$.
Let $\mathcal{F}_1$ be the $\sigma$-algebra generated by the partition $\{i\}\times\mathbb{N}$, $i=1,2,...$ of $\Omega$,
$\mathcal{F}_2$ be the $\sigma$-algebra generated by the partition $\mathbb{N}\times\{j\}$, $j=1,2,...$.
Then the marginal subspace $L^p(\mathcal{F}_1)$ consists of the functions $f$ of the form $f(x,y)=\varphi(x)$ for which
$\sum_{i=1}^\infty|\varphi(i)|^p a_i<+\infty$ where the marginal probabilities $a_i=\sum_{j=1}^\infty p_{ij}$, $i=1,2,...$.
Similarly, the marginal subspace $L^p(\mathcal{F}_2)$ consists of the functions $f$ of the form $f(x,y)=\psi(y)$ for which
$\sum_{j=1}^\infty|\psi(j)|^p b_j<+\infty$ where the marginal probabilities $b_j=\sum_{i=1}^\infty p_{ij}$, $j=1,2,...$.
Questions: when (i.e., for which $\mu$) the subspace $L^p(\mathcal{F}_1)+L^p(\mathcal{F}_2)$ is closed in $L^p(\mathcal{F})$?
When the subspace $L^p(\mathcal{F}_1)+L^p(\mathcal{F}_2)$ is complemented in $L^p(\mathcal{F})$?
We don't know.
Note that from Theorem~\ref{Th:Marginal} below it follows that if there exists a number $\alpha>0$ such that $p_{ij}\geqslant\alpha a_i b_j$
for all $i,j=1,2,...$, then $L^p(\mathcal{F}_1)+L^p(\mathcal{F}_2)$ is complemented in $L^p(\mathcal{F})$ for all $p\in[1,\infty)\cup\{\infty\}$.
Even for these $(\Omega,\mathcal{F})$ there are many similar problems.
For example, one can consider marginal subspaces of functions $f$ of the form $f(x,y)=\varphi(ax+by)$, $a,b\in\mathbb{Z}$,
$f(x,y)=\varphi(x^2+y^2)$, $f(x,y)=\varphi(y/x)$, etc.
Of course, similar questions can be posed for $n$ marginal subspaces.

\subsection{Results}
In this subsection we provide a sufficient condition for marginal subspaces $L^p_0(\mathcal{F}_1),...,L^p_0(\mathcal{F}_n)$ to be linearly independent and
their sum, $L^p_0(\mathcal{F}_1)+...+L^p_0(\mathcal{F}_n)$, to be complemented in $L^p(\mathcal{F})$ (see Theorem~\ref{Th:Marginal}).

A starting point for our result is the following simple observation: if the $\sigma$-algebras $\mathcal{F}_1,...,\mathcal{F}_n$ are
pairwise independent, then the subspaces $L^p_0(\mathcal{F}_1),...,L^p_0(\mathcal{F}_n)$ are linearly independent and their sum is complemented in $L^p(\mathcal{F})$.
This follows from observation (1) in Subsection~\ref{ss:B trivial observations}.
To see this, note that the centered conditional expectation operator $\xi\mapsto E(\xi|\mathcal{F}_i)-E\xi$
is a projection onto $L^p_0(\mathcal{F}_i)$ in $L^p(\mathcal{F})$.
Denote this operator by $P_i$.
If $\xi\in L^p_0(\mathcal{F}_j)$, $j\neq i$, then, due to independence of $\mathcal{F}_i$ and $\mathcal{F}_j$, we have $P_i\xi=E\xi-E\xi=0$.
Thus we can apply observation (1) from Subsection~\ref{ss:B trivial observations}.

Now, it is natural to think that if the $\sigma$-algebras $\mathcal{F}_1,...,\mathcal{F}_n$ are pairwise "little dependent",
then the corresponding marginal subspaces will be linearly independent and their sum will be complemented in $L^p(\mathcal{F})$.
To specify the meaning of the fuzzy words "little dependent" we first present the result of Peter J.~Bickel, Ya'akov Ritov and Jon A.~Wellner
on the closedness of the sum of two marginal subspaces in $L^2(\mathcal{F})$ (see \cite[p.1332, Proof of Lemma 1]{Bickel91}).

Let $(\Omega_1,\mathcal{A},\mu_1)$ and $(\Omega_2,\mathcal{B},\mu_2)$ be two probability spaces.
Set $(\Omega,\mathcal{F})=(\Omega_1\times\Omega_2,\mathcal{A}\otimes\mathcal{B})$.
Suppose $\mu$ is a probability measure on $\mathcal{A}\otimes\mathcal{B}$ with marginals $\mu_1$ and $\mu_2$
(that is, $\mu(A\times\Omega_2)=\mu_1(A), A\in\mathcal{A}$ and $\mu(\Omega_1\times B)=\mu_2(B), B\in\mathcal{B}$).
Let $\mathcal{F}_1=\mathcal{A}\times\Omega_2=\{A\times\Omega_2\mid A\in\mathcal{A}\}$ and
$\mathcal{F}_2=\Omega_1\times\mathcal{B}=\{\Omega_1\times B\mid B\in\mathcal{B}\}$.
Then $L^2(\mathcal{F}_1)$ consists of (equivalence classes of) random variables of the form $\xi(\omega_1,\omega_2)=f(\omega_1)$ with
$f\in L^2(\Omega_1,\mathcal{A},\mu_1)$.
Similarly, $L^2(\mathcal{F}_2)$ consists of (equivalence classes of) random variables of the form $\xi(\omega_1,\omega_2)=g(\omega_2)$ with
$g\in L^2(\Omega_2,\mathcal{B},\mu_2)$.
Bickel, Ritov and Wellner showed that if there exists $\alpha>0$ such that
$$
\mu(A\times B)\geqslant\alpha\mu_1(A)\mu_2(B), A\in\mathcal{A}, B\in\mathcal{B}
$$
then the subspaces $L^2_0(\mathcal{F}_1)$ and $L^2_0(\mathcal{F}_2)$ are linearly independent (i.e., their intersection is $\{0\}$)
and their sum is closed in $L^2(\mathcal{F})$.

Now we can specify the meaning of the fuzzy words "little dependent" for two sub-$\sigma$-algebras as follows.
Let $(\Omega,\mathcal{F},\mu)$ be a probability space.
For two sub-$\sigma$-algebras $\mathcal{A},\mathcal{B}$ of $\mathcal{F}$ define the following measure of their dependence
$$
\psi'(\mathcal{A},\mathcal{B})=\inf\left\{\dfrac{\mu(A\cap B)}{\mu(A)\mu(B)} \mid A\in\mathcal{A}, B\in\mathcal{B}, \mu(A)>0, \mu(B)>0\right\}.
$$
This measure of dependence is well known (see, e.g., \cite{Bradley05}).
It is easily seen that $0\leqslant \psi'(\mathcal{A},\mathcal{B})\leqslant 1$ and $\psi'(\mathcal{A},\mathcal{B})=1$ if and only if
$\mathcal{A}$ and $\mathcal{B}$ are independent.
Hence, using the coefficient $\psi'$, we can say that $\mathcal{A}$ and $\mathcal{B}$ are "little dependent" if
the number $1-\psi'(\mathcal{A},\mathcal{B})$ is "small".

Let us formulate our result.
Let $\mathcal{F}_1,...,\mathcal{F}_n$ be sub-$\sigma$-algebras of $\mathcal{F}$.
Define the $n\times n$ matrix $E=(e_{ij})$ by
$$e_{ij}=
\begin{cases}
0, &\text{if $i=j$;}\\
1-\psi'(\mathcal{F}_i,\mathcal{F}_j), &\text{if $i\neq j$.}
\end{cases}
$$
It is clear that $E$ is symmetric and nonnegative.
It follows that $r(E)$, the spectral radius of $E$, is the maximum eigenvalue of $E$.

\begin{theorem}\label{Th:Marginal}
If $r(E)<1$, then the marginal subspaces $L^p_0(\mathcal{F}_1),...,L^p_0(\mathcal{F}_n)$ are linearly independent
and their sum is complemented in $L^p(\mathcal{F})$.
\end{theorem}

\subsection{On the necessity of the condition $r(E)<1$}
The natural question arises:
is $r(E)<1$ a \emph{sharp} sufficient condition for $L^p_0(\mathcal{F}_1)+...+L^p_0(\mathcal{F}_n)$ to be complemented in $L^p(\mathcal{F})$?
We don't know.
We have the following conjecture (which implies that the answer is positive).

\textbf{Conjecture.}
Let $E=(e_{ij})$ be a symmetric $n\times n$ matrix with $e_{ii}=0$ for $i=1,...,n$ and $e_{ij}\geqslant 0$ for every pair $i\neq j$, $i,j\in\{1,...,n\}$.
If $r(E)=1$, then there exist a probability space $(\Omega,\mathcal{F},\mu)$ and sub-$\sigma$-algebras $\mathcal{F}_1,...,\mathcal{F}_n$ of $\mathcal{F}$
such that $\psi'(\mathcal{F}_i,\mathcal{F}_j)=1-e_{ij}$ for every pair $i\neq j$ and
$L^p_0(\mathcal{F}_1)+...+L^p_0(\mathcal{F}_n)$ is \emph{not closed} in $L^p(\mathcal{F})$ for \emph{arbitrary} $p\in[1,\infty)\cup\{\infty\}$.

\subsection{Proof of Theorem \ref{Th:Marginal}}
To prove Theorem~\ref{Th:Marginal} we will use Theorem~\ref{Th:Banach}.
The centered conditional expectation operator $\xi\mapsto E(\xi|\mathcal{F}_i)-E\xi$ is a projection onto $L^p_0(\mathcal{F}_i)$ in $L^p(\mathcal{F})$.
Denote this operator by $P_i$.
For $\xi\in L^p_0(\mathcal{F}_j)$ we have $P_i\xi=E(\xi|\mathcal{F}_i)$.
Now Theorem~\ref{Th:Marginal} is a direct consequence of Theorem~\ref{Th:Banach} and the following lemma.

\begin{lemma}
Let $\mathcal{A},\mathcal{B}$ be sub-$\sigma$-algebras of $\mathcal{F}$ and $p\in [1,\infty)\cup\{\infty\}$.
Then
\begin{equation}\label{eq:marginal ineq}
\|E(\xi|\mathcal{B})\|_p\leqslant (1-\psi'(\mathcal{A},\mathcal{B}))\|\xi\|_p,\quad \xi\in L^p_0(\mathcal{A}).
\end{equation}
\end{lemma}
\begin{remark}
For the case $(\Omega,\mathcal{F})=(\Omega_1,\Sigma_1)\otimes(\Omega_2,\Sigma_2)$, $\mathcal{A}=\Sigma_1\times\Omega_2$,
$\mathcal{B}=\Omega_1\times\Sigma_2$ inequality~\eqref{eq:marginal ineq}

(1) for $p=2$ follows from \cite[p.1332, Proof of Lemma 1]{Bickel91};

(2) for $p=\infty$ is proved in \cite[p.1331]{Bickel91}.
\end{remark}
\begin{proof}
Set $c=\psi'(\mathcal{A},\mathcal{B})$.

\textbf{1.}
Consider the probability space $(\Omega,\mathcal{F},\mu)$, the measurable space $(\Omega\times\Omega,\mathcal{A}\otimes\mathcal{B})$ and
a mapping $T:\Omega\to\Omega\times\Omega$ defined by $T\omega=(\omega,\omega), \omega\in\Omega$.
Since $T^{-1}(A\times B)=A\cap B$ for $A\in\mathcal{A}, B\in\mathcal{B}$, we see that $T$ is measurable.
Thus we can define the pushforward measure $\nu=T_{\ast}\mu$ on $\mathcal{A}\otimes\mathcal{B}$.
The measure $\nu$ has the following properties.

Firstly, we have the change-in-variables formula: if a function $f:\Omega\times\Omega\to\mathbb{K}$ is $\mathcal{A}\otimes\mathcal{B}$ measurable, then
$$
\int_{\Omega\times\Omega}f(x,y)d\nu(x,y)=\int_{\Omega}f(\omega,\omega)d\mu(\omega)
$$
(the equality means that the first integral exists if and only if the second exists, and if they exist, then they are equal).

To formulate the second property of $\nu$ denote by $\mu_1$ the restriction of $\mu$ to $\mathcal{A}$ and $\mu_2$ the restriction of $\mu$ to $\mathcal{B}$.
Then
$$
\nu(A\times B)=\mu(A\cap B)\geqslant c\mu(A)\mu(B)=c\mu_1(A)\mu_2(B)=c(\mu_1\otimes\mu_2)(A\times B)
$$
for $A\in\mathcal{A}$, $B\in\mathcal{B}$.
It follows that $\nu\geqslant c(\mu_1\otimes\mu_2)$.
Hence $\nu-c(\mu_1\otimes\mu_2)$ is a measure on $\mathcal{A}\otimes\mathcal{B}$.

\textbf{2.}
In what follows we will often use the following simple facts.
If a random variable $\xi$ is $\mathcal{A}$-measurable, then $\int_\Omega \xi d\mu=\int_{\Omega}\xi d\mu_1$
(the equality means that the first integral exists if and only if the second exists, and if they exist, then they are equal).
It follows that $\|\xi\|_{L^p(\Omega,\mathcal{F},\mu)}=\|\xi\|_{L^p(\Omega,\mathcal{A},\mu_1)}$ for $p\in[1,\infty)$.
Also, one can easily check that $\|\xi\|_{L^\infty(\Omega,\mathcal{F},\mu)}=\|\xi\|_{L^\infty(\Omega,\mathcal{A},\mu_1)}$.

Of course, similar facts are valid for the $\sigma$-algebra $\mathcal{B}$.

\textbf{3.}
Let $q\in[1,\infty)\cup\{\infty\}$ be such that $1/p+1/q=1$.
We will prove that if a random variable $\xi$ is $\mathcal{A}$-measurable and belongs to $L^p_0(\mathcal{F})$
and a random variable $\eta$ is $\mathcal{B}$-measurable and belongs to $L^q(\mathcal{F})$, then
\begin{equation}\label{eq:cov ineq}
|E(\xi\eta)|\leqslant (1-c)\|\xi\|_p\|\eta\|_q.
\end{equation}
First assume that $p\in(1,\infty)$.
Then $q\in(1,\infty)$.
Since $E\xi=0$, by the Fubini theorem we have $\int_{\Omega\times\Omega}\xi(x)\eta(y) d(\mu_1\otimes\mu_2)=0$.
Therefore
\begin{align*}
|E(\xi\eta)|&=|\int_{\Omega}\xi(\omega)\eta(\omega) d\mu|=|\int_{\Omega\times\Omega}\xi(x)\eta(y) d\nu|=
|\int_{\Omega\times\Omega}\xi(x)\eta(y) d(\nu-c(\mu_1\otimes\mu_2))|\leqslant\\
&\leqslant
\left(\int_{\Omega\times\Omega}|\xi(x)|^p d(\nu-c(\mu_1\otimes\mu_2))\right)^{1/p}
\left(\int_{\Omega\times\Omega}|\eta(y)|^q d(\nu-c(\mu_1\otimes\mu_2))\right)^{1/q}.
\end{align*}
For the first integral we have
\begin{align*}
&\int_{\Omega\times\Omega}|\xi(x)|^p d(\nu-c(\mu_1\otimes\mu_2))=
\int_{\Omega\times\Omega}|\xi(x)|^p d\nu-c\int_{\Omega\times\Omega}|\xi(x)|^p d(\mu_1\otimes\mu_2)=\\
&=\int_{\Omega}|\xi(\omega)|^p d\mu-c\int_{\Omega}|\xi(x)|^p d\mu_1=(1-c)\int_{\Omega}|\xi(\omega)|^p d\mu.
\end{align*}
Similarly, for the second integral we have
$$
\int_{\Omega\times\Omega}|\eta(y)|^q d(\nu-c(\mu_1\otimes\mu_2))=(1-c)\int_{\Omega}|\eta(\omega)|^q d\mu.
$$
It follows that
$$
|E(\xi\eta)|\leqslant (1-c)^{1/p}\|\xi\|_p (1-c)^{1/q}\|\eta\|_q=(1-c)\|\xi\|_p\|\eta\|_q.
$$
Let us prove inequality~\eqref{eq:cov ineq} for $p=1$ and $q=\infty$.
We have
\begin{align*}
|E(\xi\eta)|&=|\int_{\Omega}\xi(\omega)\eta(\omega) d\mu|=|\int_{\Omega\times\Omega}\xi(x)\eta(y) d\nu|=
|\int_{\Omega\times\Omega}\xi(x)\eta(y) d(\nu-c(\mu_1\otimes\mu_2))|\leqslant\\
&\leqslant
\int_{\Omega\times\Omega}|\xi(x)||\eta(y)| d(\nu-c(\mu_1\otimes\mu_2)).
\end{align*}
One can easily check that $|\eta(y)|\leqslant\|\eta\|_\infty$ for $(\nu-c(\mu_1\otimes\mu_2))$-almost every pair $(x,y)$.
Then
$$
|E(\xi\eta)|\leqslant\|\eta\|_\infty\int_{\Omega\times\Omega}|\xi(x)| d(\nu-c(\mu_1\otimes\mu_2))=
(1-c)\|\xi\|_1\|\eta\|_\infty.
$$
For $p=\infty$ and $q=1$ the proof of~\eqref{eq:cov ineq} is similar.

\textbf{4.}
Now we are ready to prove~\eqref{eq:marginal ineq}.
Let $\xi$ be an $\mathcal{A}$-measurable random variable which belongs to $L^p_0(\mathcal{F})$.
For each $\mathcal{B}$-measurable random variable $\eta$ which belongs to $L^q(\mathcal{F})$ we have
$$
|E(E(\xi|\mathcal{B})\eta)|=|E(\xi\eta)|\leqslant (1-c)\|\xi\|_p\|\eta\|_q.
$$
It follows that $\|E(\xi|\mathcal{B})\|_p\leqslant (1-c)\|\xi\|_p$.
\end{proof}

\section{Sums of tensor powers of subspaces}\label{s:tensor}

\subsection{}
Let $X$ be a Banach space and $X_1,...,X_n$ be complemented subspaces of $X$.
For a natural number $m$ define $X^{\otimes m}:=X\otimes...\otimes X$ to be the tensor product of $m$ copies of $X$.
Note that $X^{\otimes m}$ is merely a vector space.
Set $X_i^{\otimes m}=X_i\otimes...\otimes X_i$, $i=1,...,n$.
Clearly, $X_i^{\otimes m}$ is a subspace of $X^{\otimes m}$.
Suppose $\alpha$ is a norm on $X^{\otimes m}$.
Denote by $(X^{\otimes m})_\alpha$ the vector space $X^{\otimes m}$ endowed with the norm $\alpha$.
Let $X^m=X^m(\alpha)$ be the completion of the space and $X_i^m$ be the closure of $X_i^{\otimes m}$ in $X^m$.

We are interested in conditions under which the sum of subspaces $X_1^m,...,X_n^m$ is complemented in $X^m$.

\subsection{Results}
Consider the following two properties for the norm $\alpha$:

(P1)
if $A_1:X\to X,...,A_m:X\to X$ are bounded linear operators, then the operator
$A_1\otimes...\otimes A_m:(X^{\otimes m})_\alpha\to(X^{\otimes m})_\alpha$ is bounded.

(P2)
if $A_1:X\to X,...,A_m:X\to X$ are bounded linear operators, then the operator
$A_1\otimes...\otimes A_m:(X^{\otimes m})_\alpha\to(X^{\otimes m})_\alpha$ is bounded and its norm is equal to $\|A_1\|\|A_2\|...\|A_m\|$.

Note that the most important tensor product norms, i.e., the Hilbert space tensor product norm (when $X$ is a Hilbert space),
the projective and injective norms have the property (P2).

\begin{theorem}\label{Th:Tensor}
Assume the norm $\alpha$ has the property (P1) and $m\geqslant n-1$.
If $X_i\cap X_j=\{0\}$ and $X_i+X_j$ is complemented in $X$ for each pair of distinct indices $i,j$, then
the subspaces $X_1^m,...,X_n^m$ are linearly independent and their sum is complemented in $X^m$.
\end{theorem}

By using Theorem~\ref{Th:Banach} one can get sufficient conditions for the subspaces $X_1^m,...,X_n^m$
to be linearly independent and their sum to be complemented in $X^m$ for a given $m\geqslant 1$ (which can be smaller than $n-1$).
Let us present such conditions for the case when $X$ is a Hilbert space (but $\alpha$ is not necessarily the Hilbert space tensor product norm).

So let $X$ be a Hilbert space and $X_1,...,X_n$ be closed subspaces of $X$.
Recall that for two closed subspaces $Y,Z$ of $X$ the cosine of the minimal angle between $Y$ and $Z$, $c_0(Y,Z)$, is defined by
$$
c_0(Y,Z)=\sup\{|\langle y,z\rangle|\mid y\in Y, \|y\|\leqslant 1, z\in Z, \|z\|\leqslant 1\},
$$
here $\langle\cdot,\cdot\rangle$ is the inner product in $X$ (see, e.g., \cite{Deutsch95}).
Define the $n\times n$ matrix $E^{(m)}=(e^{(m)}_{ij})$ by
$$e^{(m)}_{ij}=
\begin{cases}
0, &\text{if $i=j$;}\\
(c_0(X_i,X_j))^m, &\text{if $i\neq j$.}
\end{cases}
$$
It is clear that $E$ is symmetric and nonnegative.
It follows that $r(E)$, the spectral radius of $E$, is the maximum eigenvalue of $E$.

\begin{theorem}\label{Th:Tensor E}
Assume the norm $\alpha$ has the property (P2).
If $r(E^{(m)})<1$, then the subspaces $X_1^m,...,X_n^m$ are linearly independent and their sum is complemented in $X^m$.
\end{theorem}

One can get a similar result in the general Banach space setting.
The result will be presented elsewhere.

\subsection{Proof of Theorem~\ref{Th:Tensor}}

If $A_1:X\to X,...,A_m:X\to X$ are bounded linear operators, then the operator
$A_1\otimes...\otimes A_m:(X^{\otimes m})_\alpha\to(X^{\otimes m})_\alpha$ is bounded.
Thus the operator can be extended, by continuity, to the bounded operator from $X^m$ to $X^m$.
We denote this extension by $(A_1\otimes...\otimes A_m)_\alpha$.

For every pair of indices $i<j$ we know that $X_i\cap X_j=\{0\}$ and $X_i+X_j$ is complemented in $X$.
Then there exist (bounded) projections $P_{ij}$ onto $X_i$ and $P_{ji}$ onto $X_j$ such that $P_{ij}|_{X_j}=0$ and $P_{ji}|_{X_i}=0$
(see Remark~\ref{r:converse}).

To prove Theorem~\ref{Th:Tensor} we will use observation (1) in Subsection~\ref{ss:B trivial observations}.
For $i=1,...,n$ define an operator $Q_i:X^m\to X^m$ by
$$
Q_i=(P_{i1}\otimes...\otimes P_{i,i-1}\otimes P_{i,i+1}\otimes...\otimes P_{i,n}\otimes P_i\otimes...\otimes P_i)_\alpha,
$$
where $P_i$ is arbitrary (bounded) projection onto $X_i$.
It is easily seen that $Q_i$ is a projection onto $X_i^m$, $i=1,...,n$.
Moreover, since $P_{ij}|_{X_j}=0$ for $i\neq j$, we conclude that $Q_i|_{X_j^m}=0$ for $i\neq j$.
Now Theorem~\ref{Th:Tensor} follows from observation (1) in Subsection~\ref{ss:B trivial observations}.

\subsection{Proof of Theorem~\ref{Th:Tensor E}}

First recall that for two closed subspaces $Y,Z$ of $X$ we have $c_0(Y,Z)=\|P_Y P_Z\|$, where
$P_Y$ is the orthogonal projection onto $Y$ and $P_Z$ is the orthogonal projection onto $Z$ (see, e.g., \cite[Lemma 10]{Deutsch95}).

To prove Theorem~\ref{Th:Tensor E} we will use Theorem~\ref{Th:Banach}.
Let $P_i$ be the orthogonal projection onto $X_i$, $i=1,...,n$.
For $i=1,...,n$ define an operator $Q_i:X^m\to X^m$ by
$$
Q_i=(P_i\otimes P_i\otimes...\otimes P_i)_\alpha.
$$
It is easily seen that $Q_i$ is a projection onto $X_i^m$.
For every pair of distinct indices $i,j$ we have
$$
\|Q_i Q_j\|=\|(P_i\otimes...\otimes P_i)_\alpha (P_j\otimes...\otimes P_j)_\alpha\|=
\|(P_i P_j\otimes...\otimes P_i P_j)_\alpha\|=\|P_i P_j\|^m=(c_0(X_i,X_j))^m.
$$
Thus, if $u\in X_j^m$, then
$$
\|Q_i u\|=\|Q_i Q_j u\|\leqslant\|Q_i Q_j\|\|u\|=(c_0(X_i,X_j))^m\|u\|.
$$
Now Theorem~\ref{Th:Tensor E} follows from Theorem~\ref{Th:Banach}.

\section{Stability of the complementability property of the sum of linearly independent subspaces}\label{s:stability}

\subsection{}\label{ss:assertion stability}
Let $X_1,...,X_n$ be closed nonzero subspaces of a Banach space $X$.
Assume that the subspaces are linearly independent and their sum, $X_1+...+X_n$, is complemented in $X$.
We will show that if closed nonzero subspaces $X_1',...,X_n'$ are such that
$X_i'$ and $X_i$ are sufficiently ``close'' to each other for all $i=1,...,n$, then
$X_1',...,X_n'$ are also linearly independent, their sum is complemented in $X$ and, moreover, the subspaces $X_1+...+X_n$ and $X_1'+...+X_n'$
have a common complementary subspace in $X$.
We will also get quantitative versions of the assertion.

\subsection{}
To specify the meaning of the fuzzy words ``$X_i'$ and $X_i$ are sufficiently close to each other'' we recall a few standard measures of closeness
of two closed subspaces of a Banach space.
The measures are the geometric opening, ball opening, spherical opening and operator opening (see, e.g., \cite{Ostrovskii94}).
We will use the first three openings.
Let us recall their definitions.
Let $X$ be a Banach space.
For an element $x\in X$ and a subset $M$ of $X$ we denote by $\text{dist}(x,M)$ the distance from $x$ to $M$, i.e., $\inf\{\|x-y\|\mid y\in M\}$.
Let $Y$ and $Z$ be two closed nonzero subspaces of $X$.
The geometric opening from $Y$ to $Z$, $\Theta_0(Y,Z)$, is defined by
$$
\Theta_0(Y,Z)=\sup\{\text{dist}(y,Z)\mid y\in S_Y\},
$$
where $S_Y$ is the unit sphere of $Y$, i.e., $\{y\in Y\mid \|y\|=1\}$.
The ball opening from $Y$ to $Z$, $\Lambda_0(Y,Z)$, is defined by
$$
\Lambda_0(Y,Z)=\sup\{\text{dist}(y,B_Z)\mid y\in S_Y\},
$$
where $B_Z$ is the closed unit ball of $Z$, i.e., $\{z\in Z\mid \|z\|\leqslant 1\}$.
The spherical opening from $Y$ to $Z$, $\Omega_0(Y,Z)$, is defined by
$$
\Omega_0(Y,Z)=\sup\{\text{dist}(y,S_Z)\mid y\in S_Y\}.
$$
It is clear that $\Theta_0(Y,Z)\leqslant\Lambda_0(Y,Z)\leqslant\Omega_0(Y,Z)$.
Now we can say that $Y$ is close to $Z$ if the number $\Theta_0(Y,Z)$ ($\Lambda_0(Y,Z)$, $\Omega_0(Y,Z)$) is small.

\subsection{}
In this Subsection we present the scheme of the proof of the assertion from Subsection~\ref{ss:assertion stability}.
Denote by $V$ a complementary subspace for $X_1+...+X_n$ in $X$.
Then the subspaces $X_1,...,X_n,V$ are linearly independent and $X_1+...+X_n+V=X$.
Let $i\in\{1,...,n\}$.
Then the subspace $\sum_{j\neq i}X_j+V$ is closed and the subspaces $X_i$ and $\sum_{j\neq i}X_j+V$ are complementary to each other in $X$.
If the subspaces $X_i'$ and $X_i$ are sufficiently close to each other, then the subspaces
$X_i'$ and $\sum_{j\neq i}X_j+V$ will also be complementary to each other in $X$.
Denote by $P_i'$ the projection onto $X_i'$ along $\sum_{j\neq i}X_j+V$.
Now we are going to use Theorem~\ref{Th:Banach} for the subspaces $X_1',...,X_n'$ and the projections $P_1',...,P_n'$.
To do this we have to estimate $\|P_i'|_{X_j'}\|$, $i\neq j$.
Let $x\in X_j'$.
Since $P_i'|_{X_j}=0$, for arbitrary $y\in X_j$ we have $\|P_i'x\|=\|P_i'(x-y)\|\leqslant \|P_i'\|\|x-y\|$.
It follows that
$$
\|P_i'x\|\leqslant\|P_i'\|\text{dist}(x,X_j)\leqslant\|P_i'\|\Theta_0(X_j',X_j)\|x\|.
$$
Hence, we can set $\varepsilon_{ij}=\|P_i'\|\Theta_0(X_j',X_j)$, $i\neq j$, and define the $n\times n$ matrix $E=(e_{ij})$ by
$$e_{ij}=
\begin{cases}
0, &\text{if $i=j$;}\\
\varepsilon_{ij}, &\text{if $i\neq j$.}
\end{cases}
$$
If $X_i'$ and $X_i$ are sufficiently close to each other for all $i=1,...,n$, then $r(E)<1$.
Thus we can use Theorem~\ref{Th:Banach}.
The theorem implies that the subspaces $X_1',...,X_n'$ are linearly independent, their sum is complemented in $X$ and, moreover, the subspace
$$
\bigcap_{i=1}^n \ker(P_i')=\bigcap_{i=1}^n \left(\sum_{j\neq i}X_j+V\right)=V
$$
is a complement of $X_1'+...+X_n'$ in $X$.

A complete proof and quantitative versions of the assertion from Subsection~\ref{ss:assertion stability} will be given in
Subsection~\ref{ss:stability complete proof}.
For this we will need a few lemmas.
Their proofs will be given in Subsections~\ref{ss:proof l closed}, \ref{ss:proof l complementary}, \ref{ss:proof l estimates} and
\ref{ss:proof l spectral radius}.

\subsection{Auxiliary Lemmas}
\begin{lemma}\label{l:closed}
Let $X$ be a Banach space and $V_1,...,V_n$ be closed subspaces of $X$.
If $V_1,...,V_n$ are linearly independent and $V_1+...+V_n$ is closed, then $V_1+...+V_{n-1}$ is also closed.
\end{lemma}

To formulate the next two lemmas we need the notion of inclination of a closed subspace to another closed subspace of a Banach space.
Let $X$ be a Banach space, $Y,Z$ be closed nonzero subspaces of $X$.
The inclination of $Y$ to $Z$, $\delta(Y,Z)$, is defined by
$$
\delta(Y,Z)=\inf\{\text{dist}(y,Z)\mid y\in S_Y\}.
$$
It is well known, and one can easily check, that $\delta(Y,Z)>0$ if and only if $Y\cap Z=\{0\}$ and $Y+Z$ is closed.
Moreover, if $Y\cap Z=\{0\}$ and $Y+Z$ is closed, then $\delta(Y,Z)=1/\|P\|$, where $P:Y+Z\to Y+Z$ is the projection onto $Y$ along $Z$.

\begin{lemma}\label{l:complementary}
Let $X$ be a Banach space, $Y,Z$ be two closed nonzero subspaces of $X$ which are complementary to each other in $X$.
Let $Y'$ be a closed nonzero subspace of $X$.
If $\Theta_0(Y',Y)<\delta(Z,Y)$ and $\Theta_0(Y,Y')<\delta(Y,Z)$, then the subspaces $Y'$ and $Z$ are also complementary to each other in $X$.
\end{lemma}

\begin{lemma}\label{l:estimates}
Let $X$ be a Banach space, $Y,Y',Z$ be closed nonzero subspaces of $X$.
Assume that $Y$ and $Z$ are complementary to each other in $X$, $Y'$ and $Z$ are also complementary to each other in $X$.
Denote by $P'$ the projection onto $Y'$ along $Z$.
Then the following estimates for $\|P'\|$ are valid:
\begin{enumerate}
\item
if $\Theta_0(Y,Y')<\delta(Y,Z)$, then
$$
\|P'\|\leqslant\frac{1+\Theta_0(Y,Y')}{\delta(Y,Z)-\Theta_0(Y,Y')}.
$$
\item
if $\Lambda_0(Y,Y')<\delta(Y,Z)$, then
$$
\|P'\|\leqslant\frac{1}{\delta(Y,Z)-\Lambda_0(Y,Y')}.
$$
\item
if $\Theta_0(Y',Y)<\delta(Y,Z)/(\delta(Y,Z)+1)$, then
$$
\|P'\|\leqslant\frac{1}{\delta(Y,Z)-(\delta(Y,Z)+1)\Theta_0(Y',Y)}.
$$
\item
if $\Omega_0(Y',Y)<\delta(Y,Z)$, then
$$
\|P'\|\leqslant\frac{1}{\delta(Y,Z)-\Omega_0(Y',Y)}.
$$
\end{enumerate}
\end{lemma}

Lemma~\ref{l:complementary} is a quantitative version of the fact that if subspaces $Y$ and $Z$ are complementary to each other in $X$
and subspaces $Y'$ and $Y$ are sufficiently close to each other, then $Y'$ and $Z$ are also complementary to each other in $X$.
Results of this type are obtained in \cite[Theorem 2]{Gokhberg59}, \cite[Theorem 5.2]{Berkson63} and \cite[Theorem 3.1(b)]{DRW05}.
One can prove Lemma~\ref{l:complementary} by specifying the arguments in \cite[Proof of Theorem 5.2]{Berkson63}.
We will do this in Subsection~\ref{ss:proof l complementary}.
Also, Lemma~\ref{l:complementary} follows from \cite[Theorem 3.1(b) and Lemma 2.3]{DRW05}
(note that in the paper the authors use the letter $\delta$ instead of $\Theta_0$).

Denote by $P$ the projection onto $Y$ along $Z$.
In \cite[Theorem 5.2]{Berkson63} and \cite[Theorem 3.1(a)]{DRW05} estimates for $\|P'-P\|$ are obtained.
Clearly, $\|P'\|\leqslant\|P'-P\|+\|P\|$.
Using this inequality, one can get the estimate for $\|P'\|$ in Lemma~\ref{l:estimates}(1) by specifying inequality (5.4) in~\cite{Berkson63}
(one can use $\Theta_0(Y,Y')$ instead of $\theta(Y,Y')$) or by \cite[inequalities (3.2) and (2.12)]{DRW05}.
Nevertheless, it is more natural to get estimates for $\|P'\|$ directly.
We will do this in the proof of Lemma~\ref{l:estimates} in Subsection~\ref{ss:proof l estimates}.

The proofs of Lemma~\ref{l:complementary} and the estimates for $\|P'\|$ in Lemma~\ref{l:estimates}(1),(2)
are very heavily based on the arguments of Berkson~\cite[Proof of Theorem 5.2]{Berkson63}.
Items (3) and (4) of Lemma~\ref{l:estimates} are simple.

\begin{lemma}\label{l:spectral radius}
Let $E=(e_{ij})$ be an $n\times n$ matrix with
$$e_{ij}=
\begin{cases}
0, &\text{if $i=j$;}\\
a_i b_j, &\text{if $i\neq j$,}
\end{cases}
$$
where $a_i>0$, $i=1,...,n$ and $b_j\geqslant 0$, $j=1,...,n$.
Then $r(E)<1$ if and only if
\begin{equation}\label{eq:sum<1}
\sum_{i=1}^n\frac{a_i b_i}{a_i b_i+1}<1.
\end{equation}
\end{lemma}

\subsection{}\label{ss:stability complete proof}
Let $X_1,...,X_n$ be closed nonzero subspaces of a Banach space $X$.
Assume that the subspaces are linearly independent and their sum, $X_1+...+X_n$, is complemented in $X$.
We will show that if closed nonzero subspaces $X_1',...,X_n'$ are such that
$X_i'$ and $X_i$ are sufficiently close to each other for all $i=1,...,n$, then
$X_1',...,X_n'$ are also linearly independent, their sum is complemented in $X$ and, moreover, the subspaces $X_1+...+X_n$ and $X_1'+...+X_n'$
have a common complementary subspace in $X$.
By using Theorem~\ref{Th:Banach} we will get various quantitative versions of the assertion.
To get the quantitative results we have to introduce quantities which characterize linear independence of $X_1,...,X_n$ and
complementability of $X_1+...+X_n$ in $X$.
We assume that

(1) there exists a projection onto $X_1+...+X_n$ of norm at most $C$, where $C$ is a positive number;

(2) the inclination $\delta(X_i,\sum_{j\neq i}X_j)\geqslant\delta_i$, $i=1,...,n$, where $\delta_1,...,\delta_n$ are positive numbers.
(Note that $\delta(X_i,\sum_{j\neq i}X_j)>0$ for $i=1,...,n$.
Indeed the closed subspaces $X_1,...,X_n$ are linearly independent and their sum is closed.
By Lemma~\ref{l:closed} the subspace $\sum_{j\neq i}X_j$ is closed.
Thus $X_i$ and $\sum_{j\neq i}X_j$ are closed subspaces of $X$ with trivial intersection and closed sum.
It follows that $\delta(X_i,\sum_{j\neq i}X_j)>0$.)

Let $\pi:X\to X$ be a projection onto $X_1+...+X_n$ of norm at most $C$.
Set $V=\ker(\pi)$, then $V$ is a complement of $X_1+...+X_n$ in $X$.
Then the subspaces $X_1,...,X_n,V$ are linearly independent and $X_1+...+X_n+V=X$.
Let $i\in\{1,...,n\}$.
By Lemma~\ref{l:closed} the subspace $\sum_{j\neq i}X_j+V$ is closed.
Thus the subspaces $X_i$ and $\sum_{j\neq i}X_j+V$ are complementary to each other in $X$.
Suppose $X_i'$, $i=1,...,n$ are closed nonzero subspaces of $X$ such that

(A1) $\Theta_0(X_i',X_i)<\delta(\sum_{j\neq i}X_j+V,X_i)$ and

(A2) $\Theta_0(X_i,X_i')<\delta(X_i,\sum_{j\neq i}X_j+V)$

for $i=1,...,n$.

\begin{remark}
In what follows we will assume that we have estimates $\Theta_0(X_i,X_i')\leqslant\theta_i$ and $\Theta_0(X_i',X_i)\leqslant\theta_i'$,
$i=1,...,n$, where $\theta_1,...,\theta_n$ and $\theta_1',...,\theta_n'$ are nonnegative numbers.
Then (A2) will be satisfied if $\theta_i<\delta_i/C$ and (A1) will be satisfied if $\theta_i'<\delta_i/(C+\delta_i)$.
This follows from the inequalities
$\delta(X_i,\sum_{j\neq i}X_j+V)\geqslant\delta_i/C$ and $\delta(\sum_{j\neq i}X_j+V,X_i)\geqslant\delta_i/(C+\delta_i)$.
Let us prove them.
The closed subspaces $X_i$ and $\sum_{j\neq i}X_j$ are complementary to each other in $X_1+...+X_n$.
Denote by $\pi_i$ the projection onto $X_i$ along $\sum_{j\neq i}X_j$.
Then $\|\pi_i\|\leqslant 1/\delta_i$.
Define the operator $P_i=\pi_i \pi:X\to X$.
It is clear that $P_i$ is the projection onto $X_i$ along $\sum_{j\neq i}X_j+V$.
We have $\|P_i\|\leqslant\|\pi_i\|\|\pi\|\leqslant C/\delta_i$.
Therefore
$$
\delta(X_i,\sum_{j\neq i}X_j+V)=\frac{1}{\|P_i\|}\geqslant\frac{\delta_i}{C}.
$$
To estimate $\delta(\sum_{j\neq i}X_j+V,X_i)$ note that
$$
\delta(\sum_{j\neq i}X_j+V,X_i)=\frac{1}{\|I-P_i\|}\geqslant\frac{1}{\|P_i\|+1}\geqslant\frac{1}{C/\delta_i+1}=\frac{\delta_i}{C+\delta_i}.
$$
\end{remark}

From (A1), (A2) and Lemma~\ref{l:complementary} it follows that $X_i'$ and $\sum_{j\neq i}X_j+V$ are complementary to each other in $X$.
Denote by $P_i'$ the projection onto $X_i'$ along $\sum_{j\neq i}X_j+V$.
Now we are going to use Theorem~\ref{Th:Banach} for the subspaces $X_1',...,X_n'$ and the projections $P_1',...,P_n'$.
To do this we have to estimate $\|P_i'|_{X_j'}\|$, $i\neq j$.
Let $x\in X_j'$.
Since $P_i'|_{X_j}=0$, for arbitrary $y\in X_j$ we have $\|P_i'x\|=\|P_i'(x-y)\|\leqslant \|P_i'\|\|x-y\|$.
It follows that
$$
\|P_i'x\|\leqslant\|P_i'\|\text{dist}(x,X_j)\leqslant\|P_i'\|\Theta_0(X_j',X_j)\|x\|.
$$
Now suppose that we have estimates $\|P_i'\|\leqslant a_i$, $i=1,...,n$.
Recall that $\Theta_0(X_j',X_j)\leqslant\theta_j'$, $j=1,...,n$.
Then $\|P_i'x\|\leqslant a_i\theta_j'\|x\|$, $x\in X_j'$.
Hence, we can set $\varepsilon_{ij}=a_i\theta_j'$, $i\neq j$, and define the $n\times n$ matrix $E=(e_{ij})$ by
$$e_{ij}=
\begin{cases}
0, &\text{if $i=j$;}\\
\varepsilon_{ij}, &\text{if $i\neq j$.}
\end{cases}
$$
If $r(E)<1$, then by Theorem~\ref{Th:Banach}
the subspaces $X_1',...,X_n'$ are linearly independent, their sum is complemented in $X$ and, moreover, the subspace
$$
\bigcap_{i=1}^n \ker(P_i')=\bigcap_{i=1}^n \left(\sum_{j\neq i}X_j+V\right)=V
$$
is a complement of $X_1'+...+X_n'$ in $X$.
By Lemma~\ref{l:spectral radius} $r(E)<1$ if and only if
\begin{equation}\label{eq:suff cond<1}
\sum_{i=1}^n\frac{a_i\theta_i'}{a_i\theta_i'+1}<1.
\end{equation}
Let us show how inequality~\eqref{eq:suff cond<1} looks for the estimates $\|P_i'\|\leqslant a_i$ given by Lemma~\ref{l:estimates}.

\begin{example}\label{ex:example1}
We will use the estimate for $\|P_i'\|$ given by Lemma~\ref{l:estimates}(1).
To use the estimate we need to assume that $\Theta_0(X_i,X_i')<\delta(X_i,\sum_{j\neq i}X_j+V)$.
Recall that $\Theta_0(X_i,X_i')\leqslant\theta_i$ and $\delta(X_i,\sum_{j\neq i}X_j+V)\geqslant\delta_i/C$.
Thus, if $\theta_i<\delta_i/C$ then the assumption is satisfied and we have
$$
\|P_i'\|\leqslant\frac{1+\Theta_0(X_i,X_i')}{\delta(X_i,\sum_{j\neq i}X_j+V)-\Theta_0(X_i,X_i')}\leqslant
\frac{1+\theta_i}{\delta_i/C-\theta_i}=\frac{C(1+\theta_i)}{\delta_i-C\theta_i}.
$$
So $a_i=C(1+\theta_i)/(\delta_i-C\theta_i)$, $i=1,...,n$.
For these $a_i$ inequality~\eqref{eq:suff cond<1}, after simple transformations, takes the form
\begin{equation}\label{eq:suff cond1}
\sum_{i=1}^n\frac{(1+\theta_i)\theta_i'}{\delta_i+C((1+\theta_i)\theta_i'-\theta_i)}<\frac{1}{C}.
\end{equation}
Note that if $\theta_i=\theta_i'$, $i=1,...,n$, then this inequality takes the form
$$
\sum_{i=1}^n\frac{\theta_i+\theta_i^2}{\delta_i+C\theta_i^2}<\frac{1}{C}.
$$
\end{example}

\begin{example}\label{ex:example2}
We will use the estimate for $\|P_i'\|$ given by Lemma~\ref{l:estimates}(2).
To use the estimate we need to assume that $\Lambda_0(X_i,X_i')<\delta(X_i,\sum_{j\neq i}X_j+V)$.
Suppose we have estimates $\Lambda_0(X_i,X_i')\leqslant\lambda_i$, $i=1,...,n$, where $\lambda_1,...,\lambda_n$ are nonnegative numbers.
If $\lambda_i<\delta_i/C$ then the assumption is satisfied and we have
$$
\|P_i'\|\leqslant\frac{1}{\delta(X_i,\sum_{j\neq i}X_j+V)-\Lambda_0(X_i,X_i')}\leqslant
\frac{1}{\delta_i/C-\lambda_i}=\frac{C}{\delta_i-C\lambda_i}.
$$
So $a_i=C/(\delta_i-C\lambda_i)$, $i=1,...,n$.
For these $a_i$ inequality~\eqref{eq:suff cond<1}, after simple transformations, takes the form
\begin{equation}\label{eq:suff cond2}
\sum_{i=1}^n\frac{\theta_i'}{\delta_i+C(\theta_i'-\lambda_i)}<\frac{1}{C}.
\end{equation}
\end{example}

\begin{example}\label{ex:example3}
We will use the estimate for $\|P_i'\|$ given by Lemma~\ref{l:estimates}(3).
To use the estimate we need to assume that $\Theta_0(X_i',X_i)<\delta(X_i,\sum_{j\neq i}X_j+V)/(\delta(X_i,\sum_{j\neq i}X_j+V)+1)$.
Recall that $\Theta_0(X_i',X_i)\leqslant\theta_i'$ and $\delta(X_i,\sum_{j\neq i}X_j+V)\geqslant\delta_i/C$.
Thus, if $\theta_i'<\delta_i/(C+\delta_i)$ then the assumption is satisfied and we have
\begin{align*}
&\|P_i'\|\leqslant\frac{1}{\delta(X_i,\sum_{j\neq i}X_j+V)-(\delta(X_i,\sum_{j\neq i}X_j+V)+1)\Theta_0(X_i',X_i)}\leqslant\\
&\leqslant\frac{1}{\delta_i/C-(\delta_i/C+1)\theta_i'}=\frac{C}{\delta_i-(C+\delta_i)\theta_i'}.
\end{align*}
So $a_i=C/(\delta_i-(C+\delta_i)\theta_i')$, $i=1,...,n$.
For these $a_i$ inequality~\eqref{eq:suff cond<1}, after simple transformations, takes the form
\begin{equation}\label{eq:suff cond3}
\sum_{i=1}^n\frac{\theta_i'}{\delta_i(1-\theta_i')}<\frac{1}{C}.
\end{equation}
\end{example}

\begin{example}\label{ex:example4}
We will use the estimate for $\|P_i'\|$ given by Lemma~\ref{l:estimates}(4).
To use the estimate we need to assume that $\Omega_0(X_i',X_i)<\delta(X_i,\sum_{j\neq i}X_j+V)$.
Suppose we have estimates $\Omega_0(X_i',X_i)\leqslant\omega_i'$, $i=1,...,n$, where $\omega_1',...,\omega_n'$ are nonnegative numbers.
If $\omega_i'<\delta_i/C$ then the assumption is satisfied and we have
$$
\|P_i'\|\leqslant\frac{1}{\delta(X_i,\sum_{j\neq i}X_j+V)-\Omega_0(X_i',X_i)}\leqslant
\frac{1}{\delta_i/C-\omega_i'}=\frac{C}{\delta_i-C\omega_i'}.
$$
So $a_i=C/(\delta_i-C\omega_i')$, $i=1,...,n$.
For these $a_i$ inequality~\eqref{eq:suff cond<1}, after simple transformations, takes the form
\begin{equation}\label{eq:suff cond4}
\sum_{i=1}^n\frac{\theta_i'}{\delta_i-C(\omega_i'-\theta_i')}<\frac{1}{C}.
\end{equation}
\end{example}

Lastly, we note that for estimation of $\|P_i'\|$ one can use different items of Lemma~\ref{l:estimates} for different $i$
(for some $i$ one can use the estimate given by Lemma~\ref{l:estimates}(1), for some $i$ --- given by Lemma~\ref{l:estimates}(2), etc.).
Then inequality~\eqref{eq:suff cond<1} will be mix of inequalities
\eqref{eq:suff cond1}, \eqref{eq:suff cond2},\eqref{eq:suff cond3} and \eqref{eq:suff cond4}
from examples \ref{ex:example1}, \ref{ex:example2}, \ref{ex:example3} and \ref{ex:example4}, respectively.

\subsection{Proof of Lemma~\ref{l:closed}}\label{ss:proof l closed}
Let $V_1\times...\times V_n$ be the linear space of all vector-columns $(v_1,...,v_n)^t$ with $v_1\in V_1,...,v_n\in V_n$
endowed with the norm $\|(v_1,...,v_n)^t\|=\|v_1\|+...+\|v_n\|$.
Then, obviously, $V_1\times...\times V_n$ is a Banach space.
Define the sum operator $S:V_1\times...\times V_n\to V_1+...+V_n$ by
$$S(v_1,...,v_n)^t=v_1+...+v_n,\quad v_1\in V_1,...,v_n\in V_n.$$
Then $S$ is a continuous linear operator with $\ker(S)=\{0\}$ and $Ran(S)=V_1+...+V_n$.
By the Banach inverse mapping theorem $S$ is an isomorphism.
It follows that $V_1+...+V_{n-1}=S(V_1\times...\times V_{n-1}\times\{0\})$ is closed in $V_1+...+V_n$.
Thus $V_1+...+V_{n-1}$ is closed in $X$.

\subsection{Proof of Lemma~\ref{l:complementary}}\label{ss:proof l complementary}
First we will prove that $Y'\cap Z=\{0\}$ and $Y'+Z$ is closed.
To this end we will show that $\delta(Y',Z)>0$.
Let $y'\in S_{Y'}$ and $z\in Z$.
For arbitrary $y\in Y$ we have
\begin{equation*}
\|y'-z\|=\|(y-z)+(y'-y)\|\geqslant \|y-z\|-\|y'-y\|=\|z-y\|-\|y'-y\|\geqslant\delta(Z,Y)\|z\|-\|y'-y\|.
\end{equation*}
Since $y\in Y$ is arbitrary, we conclude that
\begin{equation}\label{eq:ineq1}
\|y'-z\|\geqslant\delta(Z,Y)\|z\|-\text{dist}(y',Y)\geqslant\delta(Z,Y)\|z\|-\Theta_0(Y',Y).
\end{equation}
We also have $\|y'-z\|\geqslant \|y'\|-\|z\|=1-\|z\|$.
Multiplying this inequality by $\delta(Z,Y)$ and adding to~\eqref{eq:ineq1}, we get
$(1+\delta(Z,Y))\|y'-z\|\geqslant\delta(Z,Y)-\Theta_0(Y',Y)$.
Thus $\|y'-z\|\geqslant (\delta(Z,Y)-\Theta_0(Y',Y))/(1+\delta(Z,Y))$.
It follows that $\delta(Y',Z)\geqslant(\delta(Z,Y)-\Theta_0(Y',Y))/(1+\delta(Z,Y))>0$.

Let us show that $Y'+Z=X$.
To this end we will show that $Y\subset Y'+Z$.

For simplicity of notation, set $\Theta_0=\Theta_0(Y,Y')$.
Denote by $P$ the projection onto $Y$ along $Z$, by $Q$ the projection onto $Z$ along $Y$.
Recall that $\|P\|=1/\delta(Y,Z)$.
Therefore $\Theta_0\|P\|=\Theta_0/\delta(Y,Z)<1$.
Choose arbitrary number $\eta\in(1,1/(\Theta_0\|P\|))$.
Note that for every $y\in Y$ $\text{dist}(y,Y')\leqslant\Theta_0\|y\|$.
Hence there exists $y'\in Y'$ such that $\|y-y'\|\leqslant\eta\Theta_0\|y\|$.

Now we are ready to prove that $Y\subset Y'+Z$.
Consider arbitrary $y_0\in Y$.
We will choose inductively two sequences $\{y_N\mid N\geqslant 1\}\subset Y$ and $\{y_N'\mid N\geqslant 0\}\subset Y'$ as follows.

\textbf{First step.}
There exists $y_0'\in Y'$ such that $\|y_0-y_0'\|\leqslant\eta\Theta_0\|y_0\|$.
We write
$$
y_0=y_0'+(y_0-y_0')=y_0'+Q(y_0-y_0')+P(y_0-y_0')
$$
and define $y_1=P(y_0-y_0')$.

\textbf{Second step.}
There exists $y_1'\in Y'$ such that $\|y_1-y_1'\|\leqslant\eta\Theta_0\|y_1\|$.
We write
$$
y_1=y_1'+(y_1-y_1')=y_1'+Q(y_1-y_1')+P(y_1-y_1')
$$
and define $y_2=P(y_1-y_1')$ etc.

\textbf{N-th step.}
There exists $y_{N-1}'\in Y'$ such that $\|y_{N-1}-y_{N-1}'\|\leqslant\eta\Theta_0\|y_{N-1}\|$.
We write
\begin{equation}\label{eq:y N-1}
y_{N-1}=y_{N-1}'+(y_{N-1}-y_{N-1}')=y_{N-1}'+Q(y_{N-1}-y_{N-1}')+P(y_{N-1}-y_{N-1}')
\end{equation}
and define $y_N=P(y_{N-1}-y_{N-1}')$ etc.

Thus we get two sequences $\{y_N\mid N\geqslant 1\}\subset Y$ and $\{y_N'\mid N\geqslant 0\}\subset Y'$.
By the definition of $y_N$ we have
$$
\|y_N\|\leqslant\|P\|\|y_{N-1}-y_{N-1}'\|\leqslant\eta\Theta_0\|P\|\|y_{N-1}\|.
$$
It follows that
$$
\|y_N\|\leqslant(\eta\Theta_0\|P\|)^N\|y_0\|,\quad N\geqslant 0.
$$
Since $\eta\Theta_0\|P\|<1$, we see that $y_N\to 0$ as $N\to\infty$.
From~\eqref{eq:y N-1} it follows that $y_0=\sum_{k=0}^{N-1}(y_k'+Q(y_k-y_k'))+y_N$ and, consequently,
\begin{equation}\label{eq:y 0}
y_0=\sum_{k=0}^\infty(y_k'+Q(y_k-y_k')).
\end{equation}
Thus the element $y_0$ belongs to the closure of $Y'+Z$.
Recall that we have already proved that $Y'+Z$ is closed.
Hence $y_0\in Y'+Z$.
It follows that $Y\subset Y'+Z$ and therefore $Y'+Z=X$.

\subsection{Proof of Lemma~\ref{l:estimates}}\label{ss:proof l estimates}

\begin{proof}[Proof of Lemma~\ref{l:estimates}(1)]
We will use the proof of Lemma~\ref{l:complementary}.
By \eqref{eq:y 0} and continuity of $P'$ we have $P'y_0=\sum_{k=0}^\infty y_k'$.
Since $\|y_k-y_k'\|\leqslant\eta\Theta_0\|y_k\|$, we see that
$$
\|y_k'\|\leqslant(1+\eta\Theta_0)\|y_k\|\leqslant(1+\eta\Theta_0)(\eta\Theta_0\|P\|)^k\|y_0\|,\quad k\geqslant 0.
$$
It follows that
$$
\|P'y_0\|\leqslant\sum_{k=0}^\infty\|y_k'\|\leqslant\sum_{k=0}^\infty(1+\eta\Theta_0)(\eta\Theta_0\|P\|)^k\|y_0\|=
\frac{1+\eta\Theta_0}{1-\eta\Theta_0\|P\|}\|y_0\|.
$$
Letting $\eta\to 1+$ we get
$$
\|P'y_0\|\leqslant\frac{1+\Theta_0}{1-\Theta_0\|P\|}\|y_0\|.
$$
This is true for every $y_0\in Y$.
For arbitrary $x\in X$ we have
$$
\|P'x\|=\|P'(Px+Qx)\|=\|P'Px\|\leqslant\frac{1+\Theta_0}{1-\Theta_0\|P\|}\|Px\|\leqslant\frac{(1+\Theta_0)\|P\|}{1-\Theta_0\|P\|}\|x\|.
$$
Thus
$$
\|P'\|\leqslant\frac{(1+\Theta_0)\|P\|}{1-\Theta_0\|P\|}=\frac{1+\Theta_0}{1/\|P\|-\Theta_0}=\frac{1+\Theta_0}{\delta(Y,Z)-\Theta_0}.
$$
\end{proof}

\begin{proof}[Proof of Lemma~\ref{l:estimates}(2)]
The proof is similar to the proof of Lemma~\ref{l:estimates} (1).
For simplicity of notation, set $\Lambda_0=\Lambda_0(Y,Y')$.
We have $\Lambda_0\|P\|=\Lambda_0/\delta(Y,Z)<1$.
Choose arbitrary number $\eta\in(1,1/(\Lambda_0\|P\|))$.
For $r\geqslant 0$ let $B_{Y'}(r)=\{y'\in Y'\mid \|y'\|\leqslant r\}$.
Note that for every $y\in Y$ $\text{dist}(y,B_{Y'}(\|y\|))\leqslant\Lambda_0\|y\|$.
Hence there exists $y'\in Y'$ with $\|y'\|\leqslant\|y\|$ such that $\|y-y'\|\leqslant\eta\Lambda_0\|y\|$.

Consider arbitrary $y_0\in Y$.
We will choose inductively two sequences $\{y_N\mid N\geqslant 1\}\subset Y$ and $\{y_N'\mid N\geqslant 0\}\subset Y'$
in the same way as in the proof of Lemma~\ref{l:complementary} but with the only difference:
at the $N$-th step we choose $y_{N-1}'\in Y'$ with $\|y_{N-1}'\|\leqslant\|y_{N-1}\|$ such that
$\|y_{N-1}-y_{N-1}'\|\leqslant\eta\Lambda_0\|y_{N-1}\|$.
Arguing as in the proof of Lemma~\ref{l:complementary}, one can show that
$\|y_N\|\leqslant(\eta\Lambda_0\|P\|)^N\|y_0\|$ for $N\geqslant 0$
and $y_0=\sum_{k=0}^\infty(y_k'+Q(y_k-y_k'))$.
By continuity of $P'$ we have $P'y_0=\sum_{k=0}^\infty y_k'$.
Thus
$$
\|P'y_0\|\leqslant\sum_{k=0}^\infty\|y_k'\|\leqslant\sum_{k=0}^\infty\|y_k\|\leqslant
\sum_{k=0}^\infty(\eta\Lambda_0\|P\|)^k\|y_0\|=\frac{1}{1-\eta\Lambda_0\|P\|}\|y_0\|.
$$
Letting $\eta\to 1+$ we get
$$
\|P'y_0\|\leqslant\frac{1}{1-\Lambda_0\|P\|}\|y_0\|.
$$
This is valid for arbitrary $y_0\in Y$.
Similarly to the end of the proof of Lemma~\ref{l:estimates}(1) we get
$$
\|P'\|\leqslant\frac{\|P\|}{1-\Lambda_0\|P\|}=\frac{1}{1/\|P\|-\Lambda_0}=\frac{1}{\delta(Y,Z)-\Lambda_0}.
$$
\end{proof}

\begin{proof}[Proof of Lemma~\ref{l:estimates}(3)]
Recall that $\delta(Y',Z)=1/\|P'\|$.
Therefore $\|P'\|=1/\delta(Y',Z)$.
We will estimate $\delta(Y',Z)$ from below.
Let $y'\in S_{Y'}$ and $z\in Z$.
For arbitrary $y\in Y$ we have
\begin{align*}
&\|y'-z\|=\|(y-z)+(y'-y)\|\geqslant\|y-z\|-\|y'-y\|\geqslant\delta(Y,Z)\|y\|-\|y'-y\|\geqslant\\
&\geqslant\delta(Y,Z)(\|y'\|-\|y'-y\|)-\|y'-y\|=\delta(Y,Z)-(\delta(Y,Z)+1)\|y'-y\|.
\end{align*}
It follows that
$$
\|y'-z\|\geqslant\delta(Y,Z)-(\delta(Y,Z)+1)\text{dist}(y',Y)\geqslant\delta(Y,Z)-(\delta(Y,Z)+1)\Theta_0(Y',Y).
$$
Thus $\delta(Y',Z)\geqslant\delta(Y,Z)-(\delta(Y,Z)+1)\Theta_0(Y',Y)$ and consequently
$$
\|P'\|=\frac{1}{\delta(Y',Z)}\leqslant\frac{1}{\delta(Y,Z)-(\delta(Y,Z)+1)\Theta_0(Y',Y)}.
$$
\end{proof}

\begin{proof}[Proof of Lemma~\ref{l:estimates}(4)]
We will estimate $\delta(Y',Z)$ from below.
Let $y'\in S_{Y'}$ and $z\in Z$.
For arbitrary $y\in Y$ we have
$$
\|y'-z\|=\|(y-z)+(y'-y)\|\geqslant\|y-z\|-\|y'-y\|\geqslant\delta(Y,Z)\|y\|-\|y'-y\|.
$$
If $y\in S_Y$, then we get $\|y'-z\|\geqslant\delta(Y,Z)-\|y'-y\|$.
It follows that
$$
\|y'-z\|\geqslant\delta(Y,Z)-\text{dist}(y',S_Y)\geqslant\delta(Y,Z)-\Omega_0(Y',Y).
$$
Thus $\delta(Y',Z)\geqslant\delta(Y,Z)-\Omega_0(Y',Y)$ and consequently
$$
\|P'\|=\frac{1}{\delta(Y',Z)}\leqslant\frac{1}{\delta(Y,Z)-\Omega_0(Y',Y)}.
$$
\end{proof}

\subsection{Proof of Lemma~\ref{l:spectral radius}}\label{ss:proof l spectral radius}
If $b_1=...=b_n=0$, then $E=0$, $r(E)=0$ and condition~\eqref{eq:sum<1} is satisfied.
If exactly one of the numbers $b_1,...,b_n$ is greater than $0$, then $E^2=0$, $r(E)=0$ and condition~\eqref{eq:sum<1} is satisfied.
Assume that at least two of the numbers $b_1,...,b_n$ are greater than $0$.
Since the matrix $E$ is nonnegative, we conclude that $r(E)$ is an eigenvalue of $E$.
Let us consider the equation $Ew=\alpha w$, where $\alpha>0$ and $w$ is a nonzero vector.
This equation is equivalent to $\sum_{j\neq i}a_i b_j w_j=\alpha w_i$, $i=1,...,n$.
We rewrite these equations as $\sum_{j\neq i}b_j w_j=(\alpha/a_i)w_i$, $\sum_{j=1}^n b_j w_j=(b_i+\alpha/a_i)w_i$, $i=1,...,n$.
Set $s=\sum_{j=1}^n b_j w_j$.
Then $(b_i+\alpha/a_i)w_i=s$, $w_i=a_i s/(a_i b_i+\alpha)$, $i=1,...,n$.
Substituting this into the equation defining $s$, we get
$$
\sum_{j=1}^n b_j\dfrac{a_j s}{a_j b_j+\alpha}=s.
$$
If $s=0$, then $w_i=0$, $i=1,...,n$ which is impossible.
Thus $s\neq 0$ and, consequently, we get the following equation for $\alpha$:
$$
\sum_{j=1}^n\frac{a_j b_j}{a_j b_j+\alpha}=1.
$$
Define the function $f:(0,+\infty)\to\mathbb{R}$ by $f(t)=\sum_{j=1}^n a_j b_j/(a_j b_j+t)$, $t>0$.
It is clear that $f$ is continuous and decreasing on $(0,+\infty)$.
Moreover, $\lim_{t\to+\infty}f(t)=0$ and $\lim_{t\to 0+}f(t)$ is equal to the number of $j$ for which $b_j>0$,
recall that this number is at least two.
It follows that the equation $f(t)=1$ has a unique solution.
From the arguments above it follows that $r(E)$ is the solution.
It remains to note that $r(E)<1$ if and only $f(1)<1$ which is equivalent to \eqref{eq:sum<1}.

\textbf{Acknowledgements.}
This research was supported by the Project 2017-3M from the Department of Targeted Training of 
Taras Shevchenko National University of Kyiv at the NAS of Ukraine.


\begin{thebibliography}{99}

\bibitem{Badea12}
C. Badea, S. Grivaux and V. Muller,
\textit{The rate of convergence in the method of alternating projections},
St. Petersburg Math. J.
23 no.3
(2012),
413--434.

\bibitem{BauschkeBorwein96}
H.H. Bauschke and J.M. Borwein,
\textit{On projection algorithms for solving convex feasibility problems},
SIAM Rev.
38 no. 3
(1996),
367--426.

\bibitem{Berkson63}
E. Berkson,
\textit{Some metrics on the subspaces of a Banach space},
Pacific J. Math.
13 no.1
(1963),
7--22.

\bibitem{Bickel91}
P.J. Bickel, Y. Ritov and J.A. Wellner,
\textit{Efficient estimation of linear functionals of a probability measure P with known marginal distributions},
Ann. Statist.
19 no.3
(1991),
1316--1346.

\bibitem{Blot16}
J. Blot and P. Cieutat,
\textit{Completeness of Sums of Subspaces of Bounded Functions and Applications},
Commun. Math. Anal.
19 no.2
(2016),
43--61.

\bibitem{Bradley05}
R. C. Bradley,
\textit{Basic Properties of Strong Mixing Conditions. A Survey and Some Open Questions},
Probab. Surveys
2
(2005),
107--144.

\bibitem{Buja96}
A. Buja,
\textit{What Criterion for a Power Algorithm?},
in: H. Rieder (ed.),
Robust Statistics, Data Analysis, and Computer Intensive Methods.
Lecture Notes in Statistics,
vol. 109,
Springer, New York, NY,
1996,
49--61.

\bibitem{Combettes10}
P.L. Combettes and N.N. Reyes,
\textit{Functions with prescribed best linear approximations},
J. Approx. Theory
162 issue 5
(2010),
1095--1116.

\bibitem{Deutsch95}
F. Deutsch,
\textit{The angle between subspaces of a Hilbert space},
in: Approximation theory, Wavelets and Applications,
S.P. Singh (ed.),
Kluwer Academic Publishers,
The Netherlands,
1995,
107--130.

\bibitem{DRW05}
G. Dirr, V. Rako\v{c}evi\'{c} and H.K. Wimmer,
\textit{Estimates for projections in Banach spaces and existence of direct complements},
Studia Math.
170 no.2
(2005),
211--216.

\bibitem{Dixon00}
P.G. Dixon,
\textit{Non-closed sums of closed ideals in Banach algebras},
Proc. Amer. Math. Soc.
128 no.12
(2000),
3647--3654.

\bibitem{Dudziak00}
J. Dudziak, T.W. Gamelin and P. Gorkin,
\textit{Hankel operators on bounded analytic functions},
Trans. Amer. Math. Soc.
352 no.1
(2000),
363--377.

\bibitem{Feshchenko12}
I.S. Feshchenko,
\textit{On closeness of the sum of n subspaces of a Hilbert space},
Ukrainian Math. J.
63 issue 10
(2012),
1566--1622.

\bibitem{Gokhberg59}
I.Ts. Gokhberg and A.S. Markus,
\textit{Two theorems on the gap between subspaces of a Banach space},
Uspekhi Mat. Nauk,
14 issue 5
(1959),
135--140 (in Russian).

\bibitem{Gonzalez94}
M. Gonzalez,
\textit{On essentially incomparable Banach spaces},
Math. Z.
215
(1994),
621--629.

\bibitem{Hartz12}
M. Hartz,
\textit{Topological isomorphisms for some universal operator algebras},
J. Funct. Anal.
263 issue 11
(2012),
3564--3587.

\bibitem{HornJohnson13}
R.A. Horn and C.H. Johnson,
\textit{Matrix Analysis},
Second edition,
Cambridge University Press,
New York,
2013.

\bibitem{Kadets73}
M.I. Kadets and B.S. Mityagin,
\textit{Complemented subspaces in Banach spaces},
Russian Math. Surveys
28 no.6
(1973),
77--95.

\bibitem{Kim06}
H.O. Kim, R.Y. Kim and J.K. Lim,
\textit{Characterization of the closedness of the sum of two shift-invariant subspaces},
J. Math. Anal. Appl.
320 issue 1
(2006),
381--395.

\bibitem{LaVergne79}
A. LaVergne,
\textit{Remark on sums of complemented subspaces},
Colloq. Math.
41
(1979),
103--104.

\bibitem{Moslehian06}
M.S. Moslehian,
\textit{A survey of the complemented subspace problem},
Trends in Mathematics,
Information Center for Mathematical Sciences,
9 no.1
(2006),
91--98.

\bibitem{Yurdakul}
S. \"{O}nal and M. Yurdakul,
\textit{On sums of complemented subspaces},
in:
Mathematical Forum. Volume 7. Studies on mathematical analysis.
Vladikavkaz,
South Mathematical Institut of Vladikavkaz Scientific Center of Russian Academy of Sciences and Republic of North Ossetia-Alania,
2013,
148--152.

\bibitem{Ostrovskii94}
M.I. Ostrovskii,
\textit{Topologies on the set of all subspaces of a Banach space and related questions of Banach space geometry},
Quaest. Math.
17 no.3
(1994),
259--319.

\bibitem{Pinkus15}
A. Pinkus,
\textit{Ridge Functions} (Cambridge Tracts in Mathematics),
Cambridge: Cambridge University Press,
2015.

\bibitem{Pustylnik12}
E. Pustilnyk, S. Reich and A.J. Zaslavski,
\textit{Convergence of non-periodic infinite products of orthogonal projections and nonexpansive operators in Hilbert space},
J. Approx. Theory
164
(2012),
611--624.

\bibitem{Rudin75}
W. Rudin,
\textit{Spaces of type $H^\infty+C$},
Ann. Inst. Fourier (Grenoble)
25 no.1
(1975),
99--125.

\bibitem{Ruschendorf98}
L. R\"{u}schendorf and W. Thomsen,
\textit{Closedness of Sum Spaces and the Generalized Schr\"{o}dinger Problem},
Theory Probab. Appl.
42 no.3
(1998),
483--494.

\bibitem{Schochetman01}
I.E. Schochetman, R.L. Smith and S-K. Tsui,
\textit{On the closure of the sum of closed subspaces},
Int. J. Math. Math. Sci.
26 no.5
(2001),
257--267.

\bibitem{SheppKruskal78}
L.A. Shepp and J.B. Kruskal,
\textit{Computerized tomography: the new medical X-ray technology},
Amer. Math. Monthly
85 no.6
(1978),
420--439.

\bibitem{Svensson87}
L. Svensson,
\textit{Sums of complemented subspaces in locally convex spaces},
Ark. Mat.
25 issue 1
(1987),
147--153.

\end{thebibliography}
\end{document}